\documentclass[12pt]{amsart}
\usepackage{amssymb}
\usepackage{amsmath}
\usepackage{mathrsfs}
\usepackage{enumerate}
\usepackage{amsthm}
\usepackage{cancel}
\newtheorem{theorem}{Theorem}[section]
\newtheorem*{thmmainsamplingone}{Theorem \ref{mainsamplingone}}
\newtheorem*{propmainsamplinglbone}{Proposition \ref{mainsamplinglbone}}
\newtheorem{lemma}[theorem]{Lemma}
\newtheorem{corollary}[theorem]{Corollary}
\newtheorem{proposition}[theorem]{Proposition}
\theoremstyle{definition}
\newtheorem{problem}{Problem}
\newtheorem*{remark}{Remark}

\numberwithin{equation}{section}
\usepackage{color}
\usepackage{hyperref}
\hypersetup{
    colorlinks=true, %set true if you want colored links
    linkcolor=blue,  %choose some color if you want links to stand out
}
\mathchardef\hyphen="2D
\begin{document}
\allowdisplaybreaks
\title{Sharp bounds for max-sliced Wasserstein distances}
\author{March T.~Boedihardjo}
\address{Department of Mathematics, Michigan State University, East Lansing, MI 48824}
\email{boedihar@msu.edu}
\begin{abstract}
We obtain essentially matching upper and lower bounds for the expected max-sliced 1-Wasserstein distance between a probability measure on a separable Hilbert space and its empirical distribution from $n$ samples. By proving a Banach space version of this result, we also obtain an upper bound, that is sharp up to a log factor, for the expected max-sliced 2-Wasserstein distance between a symmetric probability measure $\mu$ on a Euclidean space and its symmetrized empirical distribution in terms of the operator norm of the covariance matrix of $\mu$ and the diameter of the support of $\mu$.
\end{abstract}
\maketitle
\section{Introduction}
Suppose that $\mu$ is a probability measure on $\mathbb{R}^{d}$ with $\int_{\mathbb{R}^{d}}\|x\|_{2}^{2}\,d\mu(x)<\infty$, where $\|\,\|_{2}$ is the Euclidean norm on $\mathbb{R}^{d}$. Let $X_{1},\ldots,X_{n}$ be i.i.d.~samples of $\mu$, i.e., $X_{1},\ldots,X_{n}$ are independent random vectors in $\mathbb{R}^{d}$ and each $X_{i}$ has distribution $\mu$. For a given $x\in\mathbb{R}^{d}$, the probability measure on $\mathbb{R}^{d}$ with an atom of mass $1$ is denoted by $\delta_{x}$. How many samples are needed so that the empirical distribution $\frac{1}{n}\sum_{i=1}^{n}\delta_{X_{i}}$ (which is a random probability measure on $\mathbb{R}^{d}$) is ``close" to $\mu$? Obviously the answer depends on the notion of ``close" we use. If we want the covariance matrix of $\frac{1}{n}\sum_{i=1}^{n}\delta_{X_{i}}$ to be close, in the operator norm, to the covariance matrix of $\mu$, it is already a very deep question of how many samples are needed, though by now, in some aspects, this question has been settled after a series of work \cite{Rudelson, Adamczak, Adamczak2, Romansc, SVCov, Mendelson, Koltchinskii, Tikhomirov, Zhivotovskiy, AZ}. For example, when $\mu$ is isotropic and supported on a Euclidean ball of radius $O(\sqrt{d})$, then $O(d\log d)$ samples suffice to accurately approximate the covariance matrix of $\mu$ \cite{Rudelson}. On the other hand, if we want $\frac{1}{n}\sum_{i=1}^{n}\delta_{X_{i}}$ and $\mu$ to be close in the Wasserstein distance, we need $n$ to be exponentially large in $d$ (see, e.g., \cite{Fournier}).

To circumvent this curse of dimensionality issue, in recent years, the notions of sliced Wasserstein distance \cite{Rabin, Bonneel, Carriere, Deshpande1, Kolouri1, Nadjahi, Wu}, max-sliced Wasserstein distance \cite{Deshpande2, Bartl1, Nietert, Olea} and projection robust Wasserstein distance \cite{Paty, Lin2020, Lin, WGX, NR} and other variants of the notion of sliced Wasserstein distance \cite{NHAm, Bonet, Quellmalz, NHEnergy} have been introduced, studied and used in various applications. For $p\geq 1$, the max-sliced $p$-Wasserstein distance between two probability measures $\mu$ and $\widetilde{\mu}$ on $\mathbb{R}^{d}$ is
\begin{equation}\label{maxsliceddef}
W_{p,1}(\mu,\widetilde{\mu}):=\sup_{v\in\mathbb{R}^{d},\,\|v\|_{2}=1}W_{p}(v_{\#}\mu,v_{\#}\widetilde{\mu}),
\end{equation}
where $v_{\#}\mu$ is the pushforward probability measure of $\mu$ by the map $\langle\cdot,v\rangle$, i.e., if $\mu$ is the distribution of a random vector $X$ in $\mathbb{R}^{d}$, then $v_{\#}\mu$ is the distribution of the random variable $\langle X,v\rangle$. The quantity $W_{p}(v_{\#}\mu,v_{\#}\widetilde{\mu})$ denotes the $p$-Wasserstein distance between the measures $v_{\#}\mu$ and $v_{\#}\widetilde{\mu}$ on $\mathbb{R}$. The sliced Wasserstein distance (which we do not study in this paper) is the notion where in (\ref{maxsliceddef}), we replace the supremum over $v$ by the integral of $W_{p}(v_{\#}\mu,v_{\#}\widetilde{\mu})^{p}$ over $v$ on the unit sphere and then take the $p$th root. The projection robust Wasserstein distance $W_{p,s}$ (which we also study in this paper) is the notion where in (\ref{maxsliceddef}), we take the $p$-Wasserstein distance between the pushforward measures of $\mu$ and $\widetilde{\mu}$ by a projection onto a subspace of a fixed dimension $s$ and then take supremum over all such subspaces. When $s=1$, this is the max-sliced Wasserstein distance $W_{p,1}$.

In practice (see \cite[Section 1.1]{Nietert} and the references therein), for a given probability measure $\mu$ on $\mathbb{R}^{d}$, one needs to take i.i.d.~samples $X_{1},\ldots,X_{n}$ of $\mu$ and then use the empirical measure $\mu_{n}=\frac{1}{n}\sum_{i=1}^{n}\delta_{X_{i}}$ to approximate the true measure $\mu$ in the max-sliced Wassserstein distance $W_{p,1}$. Thus for computational and statistical purposes, it is essential to know the number of samples needed to obtain an accurate approximation.

In this paper, we study the problem of finding the number of samples needed for $\mu$ and $\mu_{n}=\frac{1}{n}\sum_{i=1}^{n}\delta_{X_{i}}$ to be close in the max-sliced Wassserstein distance $W_{p,1}$.

\subsection{Max-sliced 1-Wasserstein distance}
When $p=1$, by the Kantorovich-Rubinstein theorem, the max-sliced $1$-Wasserstein distance between two probability measures $\mu$ and $\widetilde{\mu}$ on $\mathbb{R}^{d}$ coincides with the following quantity:
\begin{equation}\label{w11alt}
W_{1,1}(\mu,\widetilde{\mu})=\sup_{\substack{v\in\mathbb{R}^{d},\,\|v\|_{2}=1\\f\text{ is 1-Lipschitz}}}\left|\int_{\mathbb{R}^{d}}f(\langle x,v\rangle)\,d\mu(x)-\int_{\mathbb{R}^{d}}f(\langle x,v\rangle)\,d\widetilde{\mu}(x)\right|,
\end{equation}
where the supremum is over all the $v$ on the unit sphere and over all the 1-Lipschitz functions $f:\mathbb{R}\to\mathbb{R}$ (i.e., $|f(x)-f(y)|\leq|x-y|$ for all $x,y\in\mathbb{R}$). Consider the following problem:
\begin{problem}\label{prob1}
Suppose that $\mu$ is a probability measure on $\mathbb{R}^{d}$. Let $X_{1},\ldots,X_{n}$ be i.i.d.~samples of $\mu$. Estimate $\mathbb{E}W_{1,1}(\mu,\frac{1}{n}\sum_{i=1}^{n}\delta_{X_{i}})$.
\end{problem}

There are known estimates (some of which are sharp) of $\mathbb{E}W_{1,1}(\mu,\frac{1}{n}\sum_{i=1}^{n}\delta_{X_{i}})$ under certain regularity assumptions on the measure $\mu$, e.g., log-concavity of $\mu$ \cite[Theorem 1]{Nietert} and \cite[Theorem 1.6]{Bartl1}, or $\mu$ satisfying the transport inequality \cite[Proposition 8]{NR}, or $\mu$ satisfying the projection Bernstein tail condition or the projection Poincar\'e inequality \cite[Theorem 3.5 and Theorem 3.6]{Lin}, or $\mu$ being isotropic with its marginal distributions having uniformly bounded 4th moments \cite[Proposition 4.1]{Bartl1} (see also \cite[Remark 4.2]{Bartl1}).

As for the most general setting, under the only assumption of $\mu$ being supported on $\{x\in\mathbb{R}^{d}:\,\|x\|_{2}\leq r\}$, it was shown in \cite[Proposition 1]{Nietert} that $\mathbb{E}W_{1,1}(\mu,\frac{1}{n}\sum_{i=1}^{n}\delta_{X_{i}})\leq C\cdot\frac{rd}{\sqrt{n}}$, where $C\geq 1$ is a universal constant. In \cite[Theorem 2]{Olea}, this was improved to
$\mathbb{E}W_{1,1}(\mu,\frac{1}{n}\sum_{i=1}^{n}\delta_{X_{i}})\leq C\cdot\frac{r\sqrt{d}}{\sqrt{n}}$. In these two bounds, the rate of convergence $\frac{1}{\sqrt{n}}$ is optimal in $n$, but both bounds involve the dimension $d$.

There is a dimension-free bound for $\mathbb{E}W_{1,1}(\mu,\frac{1}{n}\sum_{i=1}^{n}\delta_{X_{i}})$ that holds with the same generality. More precisely, if $\mu$ is supported on $\{x\in\mathbb{R}^{d}:\,\|x\|_{2}\leq r\}$, then $\mathbb{E}W_{1,1}(\mu,\frac{1}{n}\sum_{i=1}^{n}\delta_{X_{i}})\leq C\cdot r\cdot n^{-1/3}$, where $C\geq 1$ is a universal constant. This follows by taking $k=1$ and optimizing the $\epsilon>0$ in the term $\mathcal{J}_{n}$ in \cite[Theorem 1]{WGX}. This estimate is dimension-free but comes at the cost of slower convergence rate in $n$.

In short, the literature concerning Problem \ref{prob1} can be summarized as follows.
\begin{enumerate}
\item If $\mu$ is supported on $\{x\in\mathbb{R}^{d}:\,\|x\|_{2}\leq r\}$, then $\mathbb{E}W_{1,1}(\mu,\frac{1}{n}\sum_{i=1}^{n}\delta_{X_{i}})\leq C(d)\cdot\frac{r}{\sqrt{n}}$, where $C(d)\geq 1$ is a constant that depends only on $d$.
\item If $\mu$ is supported on $\{x\in\mathbb{R}^{d}:\,\|x\|_{2}\leq r\}$, then $\mathbb{E}W_{1,1}(\mu,\frac{1}{n}\sum_{i=1}^{n}\delta_{X_{i}})\leq C\cdot r\cdot n^{-1/3}$, where $C\geq 1$ is a universal constant.
\item If in addition, $\mu$ satisfies certain regularity assumptions, then $\mathbb{E}W_{1,1}(\mu,\frac{1}{n}\sum_{i=1}^{n}\delta_{X_{i}})\leq C\cdot\frac{r}{\sqrt{n}}$, where $C\geq 1$ is a universal constant.
\end{enumerate}
These results together suggest the following question. Does the dimension-free bound $\mathbb{E}W_{1,1}(\mu,\frac{1}{n}\sum_{i=1}^{n}\delta_{X_{i}})\leq C\cdot\frac{r}{\sqrt{n}}$, where $C\geq 1$ is a universal constant, actually hold for every $\mu$ supported on $\{x\in\mathbb{R}^{d}:\,\|x\|_{2}\leq r\}$ even without any regularity assumptions?

In the first main result of this paper, we answer this question affirmatively. We obtain essentially matching dimension-free upper and lower bounds for $\mathbb{E}W_{1,1}(\mu,\frac{1}{n}\sum_{i=1}^{n}\delta_{X_{i}})$ in the most general setting. This essentially settles Problem \ref{prob1}.
\begin{theorem}\label{mainintro}
Suppose that $\mu$ is a probability measure on $\mathbb{R}^{d}$ with $\int_{\mathbb{R}^{d}}\|x\|_{2}^{2+\delta}\,d\mu(x)<\infty$ for some $\delta>0$. Let $X_{1},\ldots,X_{n}$ be i.i.d.~random vectors in $\mathbb{R}^{d}$ sampled according to $\mu$. Then
\[\mathbb{E}W_{1,1}\left(\mu,\frac{1}{n}\sum_{i=1}^{n}\delta_{X_{i}}\right)\leq\frac{C}{\sqrt{n}}\cdot\inf_{0<\delta\leq 1}\frac{1}{\sqrt{\delta}}\left(\int_{\mathbb{R}^{d}}\|x\|_{2}^{2+\delta}\,d\mu(x)\right)^{\frac{1}{2+\delta}},\]
where $C\geq 1$ is a universal constant. If moreover, $\int_{\mathbb{R}^{d}}x\,d\mu(x)=0$, then
\[\mathbb{E}W_{1,1}\left(\mu,\frac{1}{n}\sum_{i=1}^{n}\delta_{X_{i}}\right)\geq\frac{1}{2\sqrt{2n}}\int_{\mathbb{R}^{d}}\|x\|_{2}\,d\mu(x)\]
\end{theorem}
We also obtain a version of Theorem \ref{mainintro} for probability measures on Banach spaces. Beside being a result of intrinsic interest in the study of probability in Banach spaces (see \cite{LT}), this result is essential for proving the second main result Theorem \ref{mainsamplingone} of this paper on the max-sliced 2-Wasserstein distance for probability measures on Euclidean spaces. Indeed, in proving the latter result, we will take the Banach space $E$ to be the space of all $d\times d$ matrices equipped with the operator norm. In the Banach space setting, to define the metric $W_{1,1}$, in (\ref{w11alt}), instead of taking supremum over $v$ on the unit sphere, we take supremum over all linear functionals $v^{*}\in B_{E^{*}}$, where $B_{E^{*}}$ is the unit ball of the dual space $E^{*}$ centered at the origin. See Section \ref{defsection} for the precise definition.
\begin{theorem}\label{mainintrobanach}
Suppose that $\mu$ is a probability measure on a Banach space $(E,\|\,\|)$ with separable dual $E^{*}$ and that $\int_{E}\|x\|\,d\mu(x)<\infty$. Let $X_{1},\ldots,X_{n}$ be i.i.d.~random points in $E$ sampled according to $\mu$. Then
\[\mathbb{E}W_{1,1}\left(\mu,\frac{1}{n}\sum_{i=1}^{n}\delta_{X_{i}}\right)\leq
\frac{C}{n}\mathbb{E}\left\|\sum_{i=1}^{n}g_{i}X_{i}\right\|+\frac{C\sqrt{\ln n}}{n}\cdot\mathbb{E}\sup_{v^{*}\in B_{E^{*}}}\left(\sum_{i=1}^{n}|v^{*}(X_{i})|^{2}\right)^{\frac{1}{2}},\]
where $g_{1},\ldots,g_{n}$ are i.i.d.~standard Gaussian random variables that are independent of $X_{1},\ldots,X_{n}$, and $C\geq 1$ is a universal constant. If moreover, $\int_{E}x\,d\mu(x)=0$, then
\[\mathbb{E}W_{1,1}\left(\mu,\frac{1}{n}\sum_{i=1}^{n}\delta_{X_{i}}\right)\geq\frac{1}{2n}\mathbb{E}\left\|\sum_{i=1}^{n}\epsilon_{i}X_{i}\right\|,\]
where $\epsilon_{1},\ldots,\epsilon_{n}$ are i.i.d.~uniform $\pm 1$ random variables that are independent of $X_{1},\ldots,X_{n}$.
\end{theorem}
\begin{remark}
The separability assumption on $E^{*}$ is only to ensure that $W_{1,1}(\mu,\frac{1}{n}\sum_{i=1}^{n}\delta_{X_{i}})$ is Borel measurable.
\end{remark}
\begin{remark}
If we fix $X_{1},\ldots,X_{n}$, the quantity in the last term $\sup_{v^{*}\in B_{E^{*}}}\left(\sum_{i=1}^{n}|v^{*}(X_{i})|^{2}\right)^{\frac{1}{2}}$ is exactly the Lipschitz constant of the function $(g_{1},\ldots,g_{n})\mapsto\|\sum_{i=1}^{n}g_{i}X_{i}\|$ with respect to the Euclidean norm on $\mathbb{R}^{n}$. Moreover, by the lower bound in Khintchine's inequality \cite{Haagerup}, namely, $\mathbb{E}|\sum_{i=1}^{n}\epsilon_{i}a_{i}|\geq\frac{1}{\sqrt{2}}(\sum_{i=1}^{n}a_{i}^{2})^{\frac{1}{2}}$ for all $a_{1},\ldots,a_{n}\in\mathbb{R}$, we have
\[\mathbb{E}_{\epsilon}\left\|\sum_{i=1}^{n}\epsilon_{i}X_{i}\right\|\geq\sup_{v^{*}\in B_{E^{*}}}\mathbb{E}_{\epsilon}\left|\sum_{i=1}^{n}\epsilon_{i}v^{*}(X_{i})\right|\geq\sup_{v^{*}\in B_{E^{*}}}\frac{1}{\sqrt{2}}\left(\sum_{i=1}^{n}|v^{*}(X_{i})|^{2}\right)^{\frac{1}{2}},\]
where $\mathbb{E}_{\epsilon}$ denotes the expectation on $\epsilon_{1},\ldots,\epsilon_{n}$. Therefore, since $\mathbb{E}\|\sum_{i=1}^{n}g_{i}X_{i}\|\leq C\sqrt{\ln n}\cdot\mathbb{E}\|\sum_{i=1}^{n}\epsilon_{i}X_{i}\|$ \cite[Exercise 7.1]{vHnotes}, the upper and lower bounds in Theorem \ref{mainintrobanach} differ by at most a $C\sqrt{\ln n}$ factor.
\end{remark}
\subsection{Max-sliced 2-Wasserstein distance}
We now turn to the problem of estimating the expected max-sliced 2-Wasserstein distance $\mathbb{E}W_{2,1}(\mu,\frac{1}{n}\sum_{i=1}^{n}\delta_{X_{i}})$.

Unlike in Theorem \ref{mainintro}, for the max-sliced 2-Wasserstein distance, the convergence rate is not always the same. Even in dimension one, for certain log-concave measures $\mu$ on $\mathbb{R}$, for $p\geq 1$, the quantity $\mathbb{E}W_{p}(\mu,\frac{1}{n}\sum_{i=1}^{n}\delta_{X_{i}})$ is of order $\frac{1}{\sqrt{n}}$ \cite{Bobkov}. However, if $\mu$ is uniformly distributed on two points $1,-1\in\mathbb{R}$, one can easily see that $\mathbb{E}W_{p}(\mu,\frac{1}{n}\sum_{i=1}^{n}\delta_{X_{i}})$ is of order $n^{-1/(2p)}$, which is much slower than $\frac{1}{\sqrt{n}}$ when $p>1$.

Similarly, for the max-sliced 2-Wasserstein distance, if we assume that $\mu$ is log-concave \cite[Theorem 1(2b)]{Nietert} (see also \cite[Theorem 1.6]{Bartl1}), then $\mathbb{E}W_{2,1}(\mu,\frac{1}{n}\sum_{i=1}^{n}\delta_{X_{i}})\leq\frac{C(\mu)\log n}{\sqrt{n}}$, where $C(\mu)\geq 1$ depends only on $\mu$. On the other hand, even if $\mu$ is isotropic and its marginal distributions have uniformly bounded 4th moments, the quantity $\mathbb{E}W_{2}(\mu,\frac{1}{n}\sum_{i=1}^{n}\delta_{X_{i}})$ could already be as large as $c\cdot (d/n)^{\frac{1}{4}}$ for some universal constant $c>0$ \cite[Example 3.3]{Bartl1}.

Thus, in the most general setting (i.e., no regularity assumptions on $\mu$), the best convergence rate in $n$ for the max-sliced 2-Wasserstein distance we can hope for is $n^{-1/4}$.
\begin{corollary}\label{pboundedintro}
Let $r>0$. Suppose that $\mu$ is a probability measure on $\{x\in\mathbb{R}^{d}:\,\|x\|_{2}\leq r\}$. Let $X_{1},\ldots,X_{n}$ be i.i.d.~random vectors in $\mathbb{R}^{d}$ sampled according to $\mu$. Then for all $p\geq 1$,
\[\mathbb{E}W_{p,1}\left(\mu,\frac{1}{n}\sum_{i=1}^{n}\delta_{X_{i}}\right)\leq C\cdot r\cdot n^{-1/(2p)},\]
where $C\geq 1$ is a universal constant.
\end{corollary}
\begin{proof}
For two probability measures $\mu,\widetilde{\mu}$ on $\{x\in\mathbb{R}^{d}:\,\|x\|_{2}\leq r\}$, it is easy to see that $W_{p,1}(\mu,\widetilde{\mu})^{p}\leq (2r)^{p-1}\cdot W_{1,1}(\mu,\widetilde{\mu})$. Thus by Theorem \ref{mainintro}, the result follows.
\end{proof}
Corollary \ref{pboundedintro} removes the dimension factor in the estimate of $\mathbb{E}W_{p,1}\left(\mu,\frac{1}{n}\sum_{i=1}^{n}\delta_{X_{i}}\right)$ in \cite[Theorem 2]{Olea}.

The upper bound $C\cdot r\cdot n^{-1/(2p)}$ in Corollary \ref{pboundedintro} is attained, up to the constant $C$, when $\mu=\frac{1}{2}\delta_{y_{0}}+\frac{1}{2}\delta_{-y_{0}}$ is uniformly distributed on two points $y_{0},-y_{0}\in\mathbb{R}^{d}$ with $y_{0}$ being any vector with $\|y_{0}\|_{2}=r$.

While the bound $C\cdot r\cdot n^{-1/(2p)}$ in Corollary \ref{pboundedintro} is sharp in $n,r,p$, if one also has information on the covariance matrix of $\mu$, then perhaps, one can obtain a better bound that can depend on the covariance matrix of $\mu$. Indeed, suppose that $\mu$ is a probability measure on $\mathbb{R}^{d}$ with $\int_{\mathbb{R}^{d}}\|x\|_{2}^{2}\,d\mu(x)<\infty$. Let $\Sigma=\int_{\mathbb{R}^{d}}xx^{T}\,d\mu(x)$. Then the trace of $\Sigma$ coincides with 
\[\mathrm{Tr}(\Sigma)=\int_{\mathbb{R}^{d}}\mathrm{Tr}(xx^{T})\,d\mu(x)=\int_{\mathbb{R}^{d}}\|x\|_{2}^{2}\,d\mu(x),\]
while the operator norm of $\Sigma$ coincides with the square of the max-sliced 2-Wasserstein distance between $\mu$ and $\delta_{0}$ (the probability measure with an atom of mass $1$ at the origin). More precisely,
%Before we go into further discussions on this, we mention some simple connections between the max-sliced 2-Wasserstein distance and sample covariance matrices. The literature on sample covariance matrices gives us important intuition regarding the convergence in the max-sliced 2-Wasserstein distance.
\begin{eqnarray*}
W_{2,1}(\mu,\delta_{0})&=&\sup_{\|v\|_{2}=1}W_{2}(v_{\#}\mu,v_{\#}\delta_{0})\\&=&\sup_{\|v\|_{2}=1}\left(\int_{\mathbb{R}^{d}}|\langle x,v\rangle|^{2}\,d\mu(x)\right)^{\frac{1}{2}}=\sup_{\|v\|_{2}=1}\langle\Sigma v,v\rangle^{\frac{1}{2}}=\|\Sigma\|_{\mathrm{op}}^{\frac{1}{2}},
\end{eqnarray*}
where $\|\,\|_{\mathrm{op}}$ denotes the operator norm. Thus, for $X_{1},\ldots,X_{n}$ in $\mathbb{R}^{d}$, we have
\begin{eqnarray}\label{w2sc}
W_{2,1}\left(\mu,\frac{1}{n}\sum_{i=1}^{n}\delta_{X_{i}}\right)&\geq&W_{2,1}\left(\frac{1}{n}\sum_{i=1}^{n}\delta_{X_{i}},\delta_{0}\right)-W_{2,1}(\mu,\delta_{0})\\&=&
\left\|\frac{1}{n}\sum_{i=1}^{n}X_{i}X_{i}^{T}\right\|_{\mathrm{op}}^{\frac{1}{2}}-\|\Sigma\|_{\mathrm{op}}^{\frac{1}{2}}.\nonumber
\end{eqnarray}
So in order for $W_{2,1}\left(\mu,\frac{1}{n}\sum_{i=1}^{n}\delta_{X_{i}}\right)$ to be small, it is necessary that $\left\|\frac{1}{n}\sum_{i=1}^{n}X_{i}X_{i}^{T}\right\|_{\mathrm{op}}$ cannot be too much larger than $\|\Sigma\|_{\mathrm{op}}$.

Given that $W_{2,1}(\mu,\delta_{0})=\|\Sigma\|_{\mathrm{op}}^{\frac{1}{2}}$, the quantity $W_{2,1}(\mu,\frac{1}{n}\sum_{i=1}^{n}\delta_{X_{i}})$ should be assessed relative to $\|\Sigma\|_{\mathrm{op}}^{\frac{1}{2}}$.
\begin{problem}\label{prob2}
Suppose that $\mu$ is a probability measure on $\mathbb{R}^{d}$. Let $\Sigma=\int_{\mathbb{R}^{d}}xx^{T}\,d\mu(x)$. How many i.i.d.~samples $X_{1},\ldots,X_{n}$ of $\mu$ are needed to make the following quantity small?
\[\|\Sigma\|_{\mathrm{op}}^{-\frac{1}{2}}\cdot\mathbb{E}W_{2,1}\left(\mu,\frac{1}{n}\sum_{i=1}^{n}\delta_{X_{i}}\right).\]
\end{problem}

In \cite[Theorem 1.3]{Bartl1}, it was shown that if $\mu$ is centered and isotropic (i.e., $\Sigma=I$) with
$\sup_{v\in\mathbb{R}^{d},\,\|v\|_{2}=1}(\mathbb{E}|\langle X,v\rangle|^q)^{\frac{1}{q}}\leq L$ where $q>4$, then with high probability,
\begin{equation}\label{Bartl1result}
W_{2,1}\left(\mu,\frac{1}{n}\sum_{i=1}^{n}\delta_{X_{i}}\right)\leq C(q,L)\left[\left\|\frac{1}{n}\sum_{i=1}^{n}X_{i}X_{i}^{T}-I\right\|_{\mathrm{op}}^{\frac{1}{2}}+\left(\frac{d}{n}\right)^{\frac{1}{4}}\right],
\end{equation}
where $C(q,L)\geq 1$ is a constant that depends only on $q$ and $L$. By \cite{Tikhomirov}, the sample covariance error term $\left\|\frac{1}{n}\sum_{i=1}^{n}X_{i}X_{i}^{T}-I\right\|_{\mathrm{op}}^{\frac{1}{2}}$ is of order $\left(\frac{d}{n}\right)^{\frac{1}{4}}$ with high probability. Thus, under the assumptions mentioned above, $n=O(d)$ suffices in Problem \ref{prob2}.

The literature on sample covariance matrices (see e.g., \cite{Rudelson, Romanna, Tropp}) suggests that for a general isotropic probability measure $\mu$ supported on $\{x\in\mathbb{R}^{d}:\,\|x\|_{2}\leq C\sqrt{d}\}$ but without the assumption $\sup_{v\in\mathbb{R}^{d},\,\|v\|_{2}=1}(\mathbb{E}|\langle X,v\rangle|^q)^{\frac{1}{q}}\leq L$, the number of samples $n=O(d\log d)$ should suffice in Problem \ref{prob2}. More generally, if $\mu$ is supported on $\{x\in\mathbb{R}^{d}:\,\|x\|_{2}\leq r\}$ but not necessarily isotropic, $n=O(\frac{r^{2}}{\|\Sigma\|_{\mathrm{op}}}\log\frac{r^{2}}{\|\Sigma\|_{\mathrm{op}}})$ should suffice in Problem \ref{prob2}.

In this paper, we show that these are indeed true for symmetric $\mu$ and its symmetrized empirical distribution. A probability measure $\mu$ on $\mathbb{R}^{d}$ is symmetric if $\mu(A)=\mu(-A)$ for all measurable $A\subset\mathbb{R}^{d}$.
\begin{theorem}\label{mainsamplingone}
Let $r>0$. Suppose that $\mu$ is a symmetric probability measure on $\mathbb{R}^{d}$ supported on $\{x\in\mathbb{R}^{d}:\,\|x\|_{2}\leq r\}$. Let $X_{1},\ldots,X_{n}$ be i.i.d.~random vectors sampled according to $\mu$. Then
\[\mathbb{E}\left[W_{2,1}\left(\mu,\frac{1}{2n}\sum_{i=1}^{n}(\delta_{X_{i}}+\delta_{-X_{i}})\right)^{2}\right]\leq
C\|\Sigma\|_{\mathrm{op}}\left(\frac{r^{2}\ln n}{n\|\Sigma\|_{\mathrm{op}}}+\sqrt{\frac{r^{2}\ln n}{n\|\Sigma\|_{\mathrm{op}}}}\,\right),\]
where $\displaystyle\Sigma=\int_{\mathbb{R}^{d}}xx^{T}\,d\mu(x)$ and $C\geq 1$ is a universal constant.
\end{theorem}
The $\ln n$ factors in Theorem \ref{mainsamplingone} cannot always be removed. Indeed, consider the probability measure $\mu$ uniformly distributed on the $2d$ points $\pm \sqrt{d}\,e_{1},\ldots,\pm\sqrt{d}\,e_{d}$, where $\{e_{1},\ldots,e_{d}\}$ is the unit vector basis for $\mathbb{R}^{d}$. Then by (\ref{w2sc}), we have
\[W_{2,1}\left(\mu,\frac{1}{2n}\sum_{i=1}^{n}(\delta_{X_{i}}+\delta_{-X_{i}})\right)\geq\left\|\frac{1}{n}\sum_{i=1}^{n}X_{i}X_{i}^{T}\right\|_{\mathrm{op}}^{\frac{1}{2}}-\|\Sigma\|^{\frac{1}{2}},\]
where $\Sigma=\int_{\mathbb{R}^{d}}xx^{T}\,d\mu(x)=I$. If we view $e_{1},\ldots,e_{d}$ as $d$ bins and each $X_{i}X_{i}^{T}$ as a ball into a bin, then $\frac{1}{d}\left\|\sum_{i=1}^{n}X_{i}X_{i}^{T}\right\|$ is the maximum number of balls in a bin after $n$ balls are thrown into $d$ bins. So by \cite[Theorem 1]{Raab}, when $\frac{d}{\mathrm{polylog}(d)}\leq n\ll d\log d$,
\[\mathbb{E}\left\|\frac{1}{n}\sum_{i=1}^{n}X_{i}X_{i}^{T}\right\|_{\mathrm{op}}^{\frac{1}{2}}\geq c\left(\frac{d}{n}\cdot\frac{\log d}{\log\frac{d\log d}{n}}\right)^{\frac{1}{2}},\]
where $c>0$ is a universal constant. Thus, in this example, the $\ln n$ factors in Theorem \ref{mainsamplingone} cannot be removed.

The following lower bound result shows that the upper bound in Theorem \ref{mainsamplingone} is sharp for every covariance matrix $\Sigma$ up to the $\ln n$ factor.
\begin{proposition}\label{mainsamplinglbone}
Let $\Sigma$ be a $d\times d$ positive semidefinite matrix such that $\|\Sigma\|_{\mathrm{op}}\leq\frac{1}{2}\mathrm{Tr}(\Sigma)$. Then there exists a symmetric probability measure $\mu$ on $\mathbb{R}^{d}$ supported on $\{x\in\mathbb{R}^{d}:\,\|x\|_{2}^{2}=\mathrm{Tr}(\Sigma)\}$ such that $\int_{\mathbb{R}^{d}}xx^{T}\,d\mu(x)=\Sigma$ and for every $n\in\mathbb{N}$, 
\[\mathbb{E}\left[W_{2,1}\left(\mu,\frac{1}{2n}\sum_{i=1}^{n}(\delta_{X_{i}}+\delta_{-X_{i}})\right)^{2}\right]\geq
\frac{1}{16}\|\Sigma\|_{\mathrm{op}}\left(\frac{\mathrm{Tr}(\Sigma)}{n\|\Sigma\|_{\mathrm{op}}}+\sqrt{\frac{\mathrm{Tr}(\Sigma)}{n\|\Sigma\|_{\mathrm{op}}}}\,\right),\]
where $X_{1},\ldots,X_{n}$ are i.i.d.~random vectors sampled according to $\mu$.
\end{proposition}
\subsection{Some definitions}\label{defsection}
Throughout this paper, unless specified otherwise, we always use the Euclidean metric $\|\,\|_{2}$ on $\mathbb{R}^{d}$. If $f:\Lambda\to\mathbb{R}$ is a bounded function, then $\|f\|_{\infty}:=\sup_{x\in\Lambda}|f(x)|$. A function $f:\mathbb{R}^{s}\to\mathbb{R}$ is $1$-Lipschitz function if $|f(x)-f(y)|\leq\|x-y\|_{2}$ for all $x,y\in\mathbb{R}^{s}$. The operator norm (or equivalently the largest singular value) of a matrix $A$ is denoted by $\|A\|_{\mathrm{op}}$

If $(T,\rho)$ is a metric space and $\epsilon>0$, then the covering number $N(T,\rho,\epsilon)$ is the smallest size of $S\subset T$ for which every element of $T$ has distance at most $\epsilon$ from an element of $S$. The packing number $N_{\mathrm{pack}}(T,\rho,\epsilon)$ is the largest size of $S\subset T$ for which all elements of $S$ have distance more than $\epsilon$ away from each other. We always have
\begin{equation}\label{coverpack}
N(T,\rho,\epsilon)\leq N_{\mathrm{pack}}(T,\rho,\epsilon)\leq N\left(T,\rho,\frac{\epsilon}{2}\right).
\end{equation}
In particular, if $T_{0}\subset T$ then
\begin{equation}\label{coversubset}
N(T_{0},\rho,\epsilon)\leq N_{\mathrm{pack}}(T_{0},\rho,\epsilon)\leq N_{\mathrm{pack}}(T,\rho,\epsilon)\leq N\left(T,\rho,\frac{\epsilon}{2}\right).
\end{equation}

If $E$ is a Banach space, then the unit ball $\{x\in E:\,\|x\|\leq 1\}$ of $E$ is denoted by $B_{E}$. The dual space of all bounded linear functionals $v^{*}:E\to\mathbb{R}$ is denoted by $E^{*}$.

{\bf Pushforward measure:} If $\mu$ is a probability measure on a separable Banach space $E$ and $Q:E\to\mathbb{R}^{s}$ is a map, then $Q_{\#}\mu$ is the pushforward measure of $\mu$ by $Q$, i.e., if $X$ is a random point in $E$ with distribution $\mu$, then $Q(X)$ has distribution $Q_{\#}\mu$. In particular, $Q_{\#}\mu$ is a probability measure on $\mathbb{R}^{s}$.

{\bf Classical Wasserstein distance:} If $\mu$ and $\widetilde{\mu}$ are probability measures on $E$ and $p\geq 1$, then the $p$-Wasserstein distance between $\mu$ and $\widetilde{\mu}$ is
\[W_{p}(\mu,\widetilde{\mu}):=\inf_{\gamma}\left(\int_{E\times E}\|x-y\|^{p}\,d\gamma(x,y)\right)^{\frac{1}{p}},\]
where the infimum is over all distributions $\gamma$ on $E\times E$ with $\mu$ and $\widetilde{\mu}$ being its marginal distributions for its first and second components.

{\bf Max-sliced and projection robust Wasserstein distances:} If $\mu$ and $\widetilde{\mu}$ are probability measures on $E$ and $p\geq 1$, $s\in\mathbb{N}$, then
\[W_{p,s}(\mu,\widetilde{\mu}):=\sup_{Q}W_{p}(Q_{\#}\mu,Q_{\#}\widetilde{\mu}),\]
where the supremum is over all $Q:E\to\mathbb{R}^{s}$ of the form $Qx=(v_{1}^{*}(x),\ldots,v_{s}^{*}(x))$, for $x\in E$, with $v_{1}^{*},\ldots,v_{s}^{*}$ in the unit ball $B_{E^{*}}$ of $E^{*}$. Here we use the Euclidean distance $\|\,\|_{2}$ on $\mathbb{R}^{s}$ to define the Wasserstein distance $W_{p}$ on the right hand side.

When $p=1$, we have
\begin{align*}
&W_{1,s}(\mu,\widetilde{\mu})\\=&
\sup_{\substack{v_{1}^{*},\ldots,v_{s}^{*}\in B_{E^{*}}\\f\text{ is 1-Lipschitz}}}\left|\int_{E}f(v_{1}^{*}(x),\ldots,v_{s}^{*}(x))\,d\mu(x)-\int_{E}f(v_{1}^{*}(x),\ldots,v_{s}^{*}(x))\,d\widetilde{\mu}(x)\right|,
\end{align*}
where the supremum is over all $v_{1}^{*},\ldots,v_{s}^{*}\in B_{E^{*}}$ and all the $1$-Lipschitz functions $f:\mathbb{R}^{s}\to\mathbb{R}$.
\subsection{Organization of this paper}
In the rest of this paper, we prove the results stated in this introduction section.

In Section 2, we prove Theorem \ref{mainintro} and Theorem \ref{mainintrobanach}. The upper bound parts of Theorem \ref{mainintro} and Theorem \ref{mainintrobanach} are contained in Corollary \ref{mainsamplinghilbert} and Corollary \ref{mainsamplingbanach}, respectively. The lower bound parts of Theorem \ref{mainintro} and Theorem \ref{mainintrobanach} are stated as Corollary \ref{lbthmhilbert} and Proposition \ref{lbthm}, respectively.

In Section 3, we prove Theorem \ref{mainsamplingone} and Proposition \ref{mainsamplinglbone}.

\section{Max-sliced 1-Wasserstein distance}
In this section, we first derive a general upper bound result Theorem \ref{mainsampling} (which we obtain at a greater generality of $W_{1,s}$) for the expected max-sliced $1$-Wasserstein distance between a probability measure on a Banach space and its empirical distribution. From this result, Corollary \ref{mainsamplinghilbert} and Corollary \ref{mainsamplingbanach} follow as consequences. These give the upper bound parts of Theorem \ref{mainintro} and \ref{mainintrobanach}, respectively. Lower bound results are proved at the end of this section.

To prove Theorem \ref{mainsampling}, we use Gaussian symmetrization to reduce the problem of bounding the expected max-sliced 1-Wasserstein distance to bounding the expected supremum of a Gaussian process. To bound this expected supremum, we use Talagrand's majorizing measure theorem. We bound the metric induced by the Gaussian process by the product metric of (1) a metric on some function space (which is locally an $\|\,\|_{\infty}$ metric) and (2) a Hilbert space metric. Since Talagrand's $\gamma_{2}$ quantity of the product metric space is bounded by 3 times the sum of the $\gamma_{2}$ for each metric space, it suffices to bound the $\gamma_{2}$ for each of these two metric spaces. To bound the $\gamma_{2}$ for the first metric space, we use the Dudley's entropy integral. As for the second metric space, since it is a Hilbert space metric, the $\gamma_{2}$ for that metric space is equivalent to the supremum of some Gaussian process which, in fact, coincides with the norm of a Gaussian sum.

Let
\[N_{k}=\begin{cases}2^{2^{k}},&k\geq 1\\1,&k=0\end{cases}.\]
The following notion was introduced by Talagrand \cite{Talagrand} (see also \cite[Chapter 8]{Romanbook} and \cite[Chapter 6]{vHnotes}). For a given metric space $(T,\rho)$, define
\begin{equation}\label{gamma2}
\gamma_{2}(T,\rho):=\inf_{\mathrm{admissible}\,T_{0},T_{1},\ldots}\;\sup_{t\in T}\sum_{k=0}^{\infty}2^{\frac{k}{2}}\rho(t,T_{k}),
\end{equation}
where admissible means that $T_{0},T_{1},\ldots\subset T$ with $|T_{k}|\leq N_{k}$ for all $k\geq 0$. Also $\rho(t,T_{k})=\inf_{t_{k}\in T_{k}}\rho(t,t_{k})$.

Talagrand's majorizing measure theorem states that if $(X_{t})_{t\in T}$ is a mean zero Gaussian process, then letting $\rho(t,s)=(\mathbb{E}|X_{t}-X_{s}|^{2})^{\frac{1}{2}}$, we have
\begin{equation}\label{gc}
c\gamma_{2}(T,\rho)\leq\mathbb{E}\sup_{t\in T}X_{t}\leq C\gamma_{2}(T,\rho),
\end{equation}
where $C,c>0$ are universal constants. For the readers familiar with maximum mean discrepancy, when the kernel is given by the covariance $k(t,s)=\mathrm{cov}(X_{t},X_{s})=\mathbb{E}(X_{t}X_{s})$, the metric $\rho(t,s)$ can be thought of as the maximum mean discrepancy between the measures $\delta_{t}$ and $\delta_{s}$ on $T$. Indeed,
\begin{align*}
\mathrm{MMD}_{k}(\delta_{t},\delta_{s})^{2}=&\|k(\cdot,t)-k(\cdot,s)\|_{\mathrm{RKHS}}^{2}=k(t,t)-2k(t,s)+k(s,s)\\=&
\mathbb{E}(X_{t}^{2})-2\mathbb{E}(X_{t}X_{s})+\mathbb{E}(X_{s}^{2})=\mathbb{E}(|X_{t}-X_{s}|^{2})=\rho(t,s)^{2},
\end{align*}
where $\|\,\|_{\mathrm{RKHS}}$ denotes the reproducing kernel Hilbert space norm induced by $k(\cdot,\cdot)$.
\begin{lemma}\label{productgamma}
Let $(T,\rho_{T})$ and $(Z,\rho_{Z})$ be metric spaces. Define the metric $\rho_{T}\times\rho_{Z}$ on $T\times Z$ by
\[(\rho_{T}\times\rho_{Z})((t_{1},z_{1}),(t_{2},z_{2}))=\rho_{T}(t_{1},t_{2})+\rho_{Z}(z_{1},z_{2}).\]
Then
\[\gamma_{2}(T\times Z,\rho_{T}\times\rho_{Z})\leq3\gamma_{2}(T,\rho_{T})+3\gamma_{2}(Z,\rho_{Z}).\]
\end{lemma}
\begin{proof}
Fix $\epsilon>0$. Let $T_{0},T_{1},\ldots\subset T$ be an admissible sequence that almost attains the infimum in (\ref{gamma2}), i.e.,
\[\sup_{t\in T}\sum_{k=0}^{\infty}2^{\frac{k}{2}}\rho_{T}(t,T_{k})\leq\gamma_{2}(T,\rho_{T})+\epsilon.\]
Similarly, let $Z_{0},Z_{1},\ldots\subset Z$ be an admissible sequence such that
\[\sup_{z\in Z}\sum_{k=0}^{\infty}2^{\frac{k}{2}}\rho_{Z}(z,Z_{k})\leq\gamma_{2}(Z,\rho_{Z})+\epsilon.\]
For notational convenience, let $T_{-1}=T_{0}$ and $Z_{-1}=Z_{0}$.

Observe that the sequence $(T_{k-1}\times Z_{k-1})_{k\geq 0}$ is admissible. For all $t\in T$ and $z\in Z$, we have
\begin{eqnarray*}
\sum_{k=0}^{\infty}2^{\frac{k}{2}}(\rho_{T}\times\rho_{Z})((t,z),T_{k-1}\times Z_{k-1})&=&
\sum_{k=0}^{\infty}2^{\frac{k}{2}}[\rho_{T}(t,T_{k-1})+\rho_{Z}(z,Z_{k-1})]\\&=&
\sum_{k=-1}^{\infty}2^{\frac{k+1}{2}}\rho_{T}(t,T_{k})+\sum_{k=-1}^{\infty}2^{\frac{k+1}{2}}\rho_{Z}(z,Z_{k})\\&\leq&
3\sum_{k=0}^{\infty}2^{\frac{k}{2}}\rho_{T}(t,T_{k})+3\sum_{k=0}^{\infty}2^{\frac{k}{2}}\rho_{Z}(z,Z_{k})\\&\leq&
3[\gamma_{2}(T,\rho_{T})+\epsilon]+3[\gamma_{2}(Z,\rho_{Z})+\epsilon].
\end{eqnarray*}
So
\[\gamma_{2}((T,Z),\rho_{T}\times\rho_{Z})\leq3\gamma_{2}(T,\rho_{T})+3\gamma_{2}(Z,\rho_{Z})+6\epsilon.\]
Since this holds for all $\epsilon>0$, the result follows.
\end{proof}
\begin{lemma}[\cite{Talagrand}, page 12-13]\label{Dudleybound}
Let $(T,\rho_{T})$ be a metric space. Then
\[\gamma_{2}(T,\rho_{T})\leq C\int_{0}^{\infty}\sqrt{\log N(T,\rho_{T},\epsilon)}\,d\epsilon,\]
where $C\geq 1$ is a universal constant.
\end{lemma}
Next we bound the covering number of a set of 1-Lipschitz functions with respect to a certain norm (see Lemma \ref{fscovering} below). This will be needed when we apply Lemma \ref{Dudleybound} to bound the $\gamma_{2}$ quantity for that metric space of $1$-Lipschitz functions. Before we do that, we need a basic result.

In the sequel, the readers who are interested in the max-sliced Wasserstein distances but not the general projection robust Wasserstein distances may take $s=1$ in the rest of this paper. This will be enough to prove the main results mentioned in the introduction section.
\begin{lemma}\label{basiclipcover}
Let $a>0$. Let
\[D=\{h:[-a,a]^{s}\to\mathbb{R}|\,h\text{ is 1-Lipschitz and }h(0)=0\}.\]
Then
\[N(D,\|\,\|_{\infty},\epsilon)\leq\exp\left(\left(\frac{Ca\sqrt{s}}{\epsilon}\right)^{s}\right),\]
for all $\epsilon>0$, where $C\geq 1$ is a universal constant.
\end{lemma}
\begin{proof}
The map $h\to(x\mapsto h(ax))$ defines an isometry from the metric space $(D,\|\,\|_{\infty})$ to the metric space $(\widetilde{D},\|\,\|_{\infty})$, where
\[\widetilde{D}=\{h:[-1,1]^{s}\to\mathbb{R}|\,h\text{ is }a\text{-Lipschitz and }h(0)=0\}.\]
So
\[N(D,\|\,\|_{\infty},\epsilon)=N(\widetilde{D},\|\,\|_{\infty},\epsilon).\]
Since
\begin{align*}
\widetilde{D}\subset\widehat{D}:=\{h:[-1,1]^{s}\to\mathbb{R}:\,&h(0)=0\text{ and }\\&
|h(x)-h(y)|\leq a\sqrt{s}\max_{i}|x_{i}-y_{i}|\;\forall x,y\in[-1,1]^{s}\}
\end{align*}
and it is well known (see, e.g., \cite[page 129]{Wainwright}) that $N(\widehat{D},\|\,\|_{\infty},\epsilon)\leq\exp((\frac{Ca\sqrt{s}}{\epsilon})^{s})$, by (\ref{coversubset}), it follows that
\[N(\widetilde{D},\|\,\|_{\infty},\epsilon)\leq N\left(\widehat{D},\|\,\|_{\infty},\frac{\epsilon}{2}\right)\leq\exp\left(\left(\frac{Ca\sqrt{s}}{\epsilon/2}\right)^{s}\right).\]
So the result follows.
\end{proof}

\begin{lemma}\label{fscovering}
Let $T$ be the set of all $1$-Lipschitz functions $f:\mathbb{R}^{s}\to\mathbb{R}$ with $f(0)=0$. For $0<\delta\leq 1$, define the norm $\|\,\|_{(\delta)}$ on $T$ by
\begin{equation}\label{starnorm}
\|f\|_{(\delta)}=\sup_{x\in\mathbb{R}^{s}}\frac{|f(x)|}{\|x\|_{2}^{1+\delta}+1}.
\end{equation}
Then
\[\log N(T,\|\,\|_{(\delta)},\epsilon)\leq\left(\frac{C\sqrt{s}}{\epsilon}\right)^{s}\frac{1}{\delta},\]
for all $\epsilon>0$ and $0<\delta\leq 1$, where $C\geq 1$ is a universal constant.
\end{lemma}
\begin{proof}
Set $\Omega_{0}=\{x\in\mathbb{R}^{s}:\,\|x\|_{2}\leq 1\}$, and for $j\in\mathbb{N}$, set
\[\Omega_{j}=\{x\in\mathbb{R}^{s}:\,2^{j-1}\leq\|x\|_{2}\leq 2^{j}\}\cup\{0\}.\]
Let
\[A_{j}=\{h:\Omega_{j}\to\mathbb{R}|\,h\text{ is 1-Lipschitz and }h(0)=0\}.\]
Define the following norm $\|\,\|_{(\delta),j}$ on $A_{j}$:
\[\|h\|_{(\delta),j}=\sup_{x\in\Omega_{j}}\frac{|h(x)|}{\|x\|_{2}^{1+\delta}+1}\quad\text{for }h\in A_{j}.\]
For every $f\in T$, observe that the restriction $f|_{\Omega_{j}}\in A_{j}$ and
\[\|f\|_{(\delta)}=\sup_{j\geq 0}\left\|f|_{\Omega_{j}}\right\|_{(\delta),j}.\]
Thus, $(T,\|\,\|_{(\delta)})$ can be identified as a metric subspace of the product metric space $\prod_{j=0}^{\infty}(A_{j},\|\,\|_{(\delta),j})$. So by (\ref{coversubset}), the $\epsilon$-covering number of $(T,\|\,\|_{(\delta)})$ is bounded by the $\frac{\epsilon}{2}$-covering number of $\prod_{j=0}^{\infty}(A_{j},\|\,\|_{(\delta),j})$. Hence,
\begin{equation}\label{fscoveringeq}
N(T,\|\,\|_{(\delta)},\epsilon)\leq\prod_{j=0}^{\infty}N\left(A_{j},\|\,\|_{(\delta),j},\frac{\epsilon}{2}\right).
\end{equation}
Note that for all $j\geq 1+\frac{1}{\delta}\log_{2}\frac{2}{\epsilon}$ and $h\in A_{j}$, we have
\[\|h\|_{(\delta),j}=\sup_{x\in\Omega_{j}}\frac{|h(x)|}{\|x\|_{2}^{1+\delta}+1}\leq\sup_{x\in\Omega_{j}\backslash\{0\}}\frac{\|x\|_{2}}{\|x\|_{2}^{1+\delta}+1}\leq\sup_{x\in\Omega_{j}\backslash\{0\}}\|x\|_{2}^{-\delta}\leq 2^{-\delta(j-1)}\leq\frac{\epsilon}{2}.\]
So $N(A_{j},\|\,\|_{(\delta),j},\frac{\epsilon}{2})=1$ for all $j\geq1+\frac{1}{\delta}\log_{2}\frac{2}{\epsilon}$.

For $j\geq 0$, let
\begin{equation}\label{djdef}
D_{j}=\{h:[-2^{j},2^{j}]^{s}\to\mathbb{R}|\,h\text{ is 1-Lipschitz and }h(0)=0\}.
\end{equation}
Note that $\Omega_{j}\subset[-2^{j},2^{j}]^{s}$. Every function $h\in A_{j}$ can be extended to a function $\tau(h)\in D_{j}$ (by Kirszbraun extension), where
\[[\tau(h)](x)=\inf_{y\in\Omega_{j}}(h(y)+\|x-y\|_{2})\quad\text{for }x\in[-2^{j},2^{j}]^{s}.\]
(Note that $\tau(0)$ is not the zero function, but $[\tau(h)](0)=0$.) For all $h_{1},h_{2}\in A_{j}$ with $j\geq 0$,
\begin{eqnarray*}
\|h_{1}-h_{2}\|_{(\delta),j}&=&\sup_{x\in\Omega_{j}\backslash\{0\}}\frac{|h_{1}(x)-h_{2}(x)|}{\|x\|_{2}^{1+\delta}+1}\\&\leq&
\sup_{x\in\Omega_{j}\backslash\{0\}}\frac{|h_{1}(x)-h_{2}(x)|}{2^{(j-1)(1+\delta)}}
\\&\leq&
2^{(1-j)(1+\delta)}\sup_{x\in[-2^{j},2^{j}]^{s}}|[\tau(h_{1})](x)-[\tau(h_{2})](x)|\\&=&
2^{(1-j)(1+\delta)}\|\tau(h_{1})-\tau(h_{2})\|_{\infty},
\end{eqnarray*}
where $\displaystyle\|h\|_{\infty}=\sup_{x\in[-2^{j},2^{j}]^{s}}|h(x)|$ for $h\in D_{j}$. So for all $j\geq 0$,
\begin{eqnarray*}
N(A_{j},\|\,\|_{(\delta),j},\epsilon)&\leq&
N_{\mathrm{pack}}(A_{j},\|\,\|_{(\delta),j},\epsilon)\\&\leq& N_{\mathrm{pack}}(D_{j},\|\,\|_{\infty},2^{(j-1)(1+\delta)}\epsilon)\\&\leq&
N\left(D_{j},\|\,\|_{\infty},2^{(j-1)(1+\delta)}\frac{\epsilon}{2}\right)\text{ by }(\ref{coverpack})\\&\leq&
\exp\left(\left(\frac{C\cdot 2^{j}\sqrt{s}}{2^{(j-1)(1+\delta)}\epsilon/2}\right)^{s}\right)=\exp\left(\left(\frac{2C\sqrt{s}}{\epsilon}\right)^{s}2^{s(1+\delta-j\delta)}\right),
\end{eqnarray*}
where the last inequality follows from Lemma \ref{basiclipcover}. Therefore, by (\ref{fscoveringeq}),
\[\log N(T,\|\,\|_{(\delta)},\epsilon)\leq\sum_{j=0}^{\infty}\left(\frac{4C\sqrt{s}}{\epsilon}\right)^{s}2^{s(1+\delta-j\delta)}.\]
But
\[\sum_{j=0}^{\infty}2^{s(1+\delta-j\delta)}=\frac{2^{s(1+\delta)}}{1-2^{-s\delta}}\leq \frac{2^{2s}}{1-2^{-\delta}}\leq C\cdot\frac{2^{2s}}{\delta},\]
since $0<\delta\leq 1$. So the result follows.
\end{proof}
The following result is the main lemma of this section. We bound the expected supremum of the Gaussian process that arises when we use Gaussian symmetrization to prove Theorem \ref{mainsampling}. The key ingredient in proving this lemma is Talagrand's majorizing measure theorem.
\begin{lemma}\label{mainlemma}
Let $0<\delta\leq 1$. Suppose that $E$ is a Banach space with separable dual $E^{*}$ and $x_{1},\ldots,x_{n}\in E$. Let $g_{1},\ldots,g_{n}$ be i.i.d.~standard Gaussian random variables. Let $T$ be the set of all $1$-Lipschitz functions $f:\mathbb{R}^{s}\to\mathbb{R}$ with $f(0)=0$. Then
\begin{align*}
&\mathbb{E}\sup_{\substack{v_{1}^{*},\ldots,v_{s}^{*}\in B_{E^{*}}\\f\in T}}\left|\frac{1}{n}\sum_{i=1}^{n}g_{i}f(v_{1}^{*}(x_{i}),\ldots,v_{s}^{*}(x_{i}))\right|\\\leq&
\frac{Cs}{n}\mathbb{E}\left\|\sum_{i=1}^{n}g_{i}x_{i}\right\|+\frac{CM\sqrt{s}}{\sqrt{n}}\cdot\begin{cases}(\delta n)^{-\frac{1}{2}},&s=1\\(\ln(\delta n+2))\cdot(\delta n)^{-\frac{1}{2}},&s=2\\(\delta n)^{-\frac{1}{s}},&s\geq 3\end{cases},
\end{align*}
where
\begin{equation}\label{mdef}
M=\sqrt{2}\left(n+s^{1+\delta}\sup_{v^{*}\in B_{E^{*}}}\sum_{i=1}^{n}|v^{*}(x_{i})|^{2+2\delta}\right)^{\frac{1}{2}},
\end{equation}
and $B_{E^{*}}=\{v^{*}\in E^{*}:\,\|v^{*}\|\leq 1\}$.
\end{lemma}
\begin{proof}
Let $Z=\{(v_{1}^{*},\ldots,v_{s}^{*}):\,v_{1}^{*},\ldots,v_{s}^{*}\in B_{E^{*}}\}$. Define the Gaussian process $(X_{f,z})_{(f,z)\in T\times Z}$ as follows. If $f\in T$ and $z=(v_{1}^{*},\ldots,v_{s}^{*})\in Z$, then
\[X_{f,z}=\sum_{i=1}^{n}g_{i}f(v_{1}^{*}(x_{i}),\ldots,v_{s}^{*}(x_{i})).\]
Recall that $\|\,\|_{(\delta)}$ is defined in (\ref{starnorm}). For $f,h\in T$ and $(v_{1}^{*},\ldots,v_{s}^{*})\in Z$, we have
\begin{align}\label{fhbound}
&\|\{f(v_{1}^{*}(x_{i}),\ldots,v_{s}^{*}(x_{i}))\}_{1\leq i\leq n}-\{h(v_{1}^{*}(x_{i}),\ldots,v_{s}^{*}(x_{i}))\}_{1\leq i\leq n}\|_{2}\\=&
\left(\sum_{i=1}^{n}|f(v_{1}^{*}(x_{i}),\ldots,v_{s}^{*}(x_{i}))-h(v_{1}^{*}(x_{i}),\ldots,v_{s}^{*}(x_{i}))|^{2}\right)^{\frac{1}{2}}\nonumber\\\leq&
\|f-h\|_{(\delta)}\left(\sum_{i=1}^{n}\left[1+\|(v_{1}^{*}(x_{i}),\ldots,v_{s}^{*}(x_{i}))\|_{2}^{1+\delta}\right]^{2}\right)^{\frac{1}{2}}\nonumber\\\leq&
\|f-h\|_{(\delta)}\left(\sum_{i=1}^{n}2\left[1+\|(v_{1}^{*}(x_{i}),\ldots,v_{s}^{*}(x_{i}))\|_{2}^{2+2\delta}\right]\right)^{\frac{1}{2}}\nonumber\\\leq&
\|f-h\|_{(\delta)}\sqrt{2}\left(\sum_{i=1}^{n}\left[1+s^{\delta}(|v_{1}^{*}(x_{i})|^{2+2\delta}+\ldots+|v_{s}^{*}(x_{i})|^{2+2\delta})\right]\right)^{\frac{1}{2}}\nonumber\\&\text{by H\"older's inequality with }p=1+\delta\nonumber\\\leq&
\|f-h\|_{(\delta)}\sqrt{2}\left(n+s^{1+\delta}\sup_{v^{*}\in B_{E^{*}}}\sum_{i=1}^{n}|v^{*}(x_{i})|^{2+2\delta}\right)^{\frac{1}{2}}\nonumber\\=&
M\|f-h\|_{(\delta)},\nonumber
\end{align}
where $M>0$ is defined in (\ref{mdef}).

Fix $b>0$. Let $T^{(b)}\subset T$ be a $b$-covering of $T$ with respect to $\|\,\|_{(\delta)}$ that has the smallest size, i.e., $|T^{(b)}|=N(T,\|\,\|_{(\delta)},b)$. %(We will see that when $s=1$, we can take $b=0$ so that $T^{(b)}=T$.)
For every $f\in T$, there exists $h\in T^{(b)}$ such that $\|f-h\|_{(\delta)}\leq b$ so by (\ref{fhbound}),
\[\|\{f(v_{1}^{*}(x_{i}),\ldots,v_{s}^{*}(x_{i}))\}_{1\leq i\leq n}-\{h(v_{1}^{*}(x_{i}),\ldots,v_{s}^{*}(x_{i}))\}_{1\leq i\leq n}\|_{2}\leq bM,\]
for all $v_{1}^{*},\ldots,v_{s}^{*}\in B_{E^{*}}$. So
\begin{align*}
&\sup_{v_{1}^{*},\ldots,v_{s}^{*}\in B_{E^{*}}}\frac{1}{n}\sum_{i=1}^{n}g_{i}f(v_{1}^{*}(x_{i}),\ldots,v_{s}^{*}(x_{i}))\\\leq&
\sup_{v_{1}^{*},\ldots,v_{s}^{*}\in B_{E^{*}}}\frac{1}{n}\sum_{i=1}^{n}g_{i}h(v_{1}^{*}(x_{i}),\ldots,v_{s}^{*}(x_{i}))+\frac{1}{n}\|(g_{1},\ldots,g_{n})\|_{2}\cdot bM.
\end{align*}
So since $T=-T$, we have
\begin{align}\label{discretize}
&\mathbb{E}\sup_{\substack{v_{1}^{*},\ldots,v_{s}^{*}\in B_{E^{*}}\\f\in T}}\left|\frac{1}{n}\sum_{i=1}^{n}g_{i}f(v_{1}^{*}(x_{i}),\ldots,v_{s}^{*}(x_{i}))\right|\\=&
\mathbb{E}\sup_{\substack{v_{1}^{*},\ldots,v_{s}^{*}\in B_{E^{*}}\\f\in T}}\frac{1}{n}\sum_{i=1}^{n}g_{i}f(v_{1}^{*}(x_{i}),\ldots,v_{s}^{*}(x_{i}))\nonumber
\\\leq&
\mathbb{E}\sup_{\substack{v_{1}^{*},\ldots,v_{s}^{*}\in B_{E^{*}}\\h\in T^{(b)}}}\frac{1}{n}\sum_{i=1}^{n}g_{i}h(v_{1}^{*}(x_{i}),\ldots,v_{s}^{*}(x_{i}))+\mathbb{E}\frac{1}{n}\|(g_{1},\ldots,g_{n})\|_{2}\cdot bM\nonumber\\\leq&
\mathbb{E}\sup_{\substack{v_{1}^{*},\ldots,v_{s}^{*}\in B_{E^{*}}\\h\in T^{(b)}}}\frac{1}{n}\sum_{i=1}^{n}g_{i}h(v_{1}^{*}(x_{i}),\ldots,v_{s}^{*}(x_{i}))+\frac{bM}{\sqrt{n}}\nonumber\\&\quad(\text{since }\mathbb{E}\|(g_{1},\ldots,g_{n})\|_{2}\leq(\mathbb{E}\|(g_{1},\ldots,g_{n})\|_{2}^{2})^{\frac{1}{2}}=\sqrt{n})\nonumber\\=&
\frac{1}{n}\mathbb{E}\sup_{(h,z)\in T^{(b)}\times Z}X_{h,z}+\frac{bM}{\sqrt{n}},\nonumber
\end{align}
where $X_{h,z}$ is defined at the beginning of this proof.

For $f,h\in T$ and $z=(v_{1}^{*},\ldots,v_{s}^{*})\in Z$, $\widetilde{z}=(w_{1}^{*},\ldots,w_{s}^{*})\in Z$, we have
\begin{align}\label{xdiffbound}
&\left(\mathbb{E}|X_{f,z}-X_{h,\widetilde{z}}|^{2}\right)^{\frac{1}{2}}\\=&
\|\{f(v_{1}^{*}(x_{i}),\ldots,v_{s}^{*}(x_{i}))\}_{1\leq i\leq n}-\{h(w_{1}^{*}(x_{i}),\ldots,w_{s}^{*}(x_{i}))\}_{1\leq i\leq n}\|_{2}\nonumber\\\leq&
\|\{f(v_{1}^{*}(x_{i}),\ldots,v_{s}^{*}(x_{i}))\}_{1\leq i\leq n}-\{h(v_{1}^{*}(x_{i}),\ldots,v_{s}^{*}(x_{i}))\}_{1\leq i\leq n}\|_{2}\nonumber\\&
+\|\{h(v_{1}^{*}(x_{i}),\ldots,v_{s}^{*}(x_{i}))\}_{1\leq i\leq n}-\{h(w_{1}^{*}(x_{i}),\ldots,w_{s}^{*}(x_{i}))\}_{1\leq i\leq n}\|_{2}\nonumber\\\leq&
M\|f-h\|_{(\delta)}+\left(\sum_{i=1}^{n}|h(v_{1}^{*}(x_{i}),\ldots,v_{s}^{*}(x_{i}))-h(w_{1}^{*}(x_{i}),\ldots,w_{s}^{*}(x_{i}))|^{2}\right)^{\frac{1}{2}}\nonumber\\\leq&
M\|f-h\|_{(\delta)}+\left(\sum_{i=1}^{n}\sum_{j=1}^{s}|v_{j}^{*}(x_{i})-w_{j}^{*}(x_{i})|^{2}\right)^{\frac{1}{2}},\nonumber
\end{align}
where the second inequality follows from (\ref{fhbound}) and the last inequality follows from $h$ being 1-Lipschitz. Recall that $M>0$ is defined in (\ref{mdef}) and $\|\,\|_{(\delta)}$ is defined in (\ref{starnorm}). Consider the metric $\rho_{T}(f,h)=M\|f-h\|_{(\delta)}$ on $T$. Also, define the metric $\rho_{Z}$ on $Z$ by
\[\rho_{Z}((v_{1}^{*},\ldots,v_{s}^{*}),(w_{1}^{*},\ldots,w_{s}^{*}))=\left(\sum_{i=1}^{n}\sum_{j=1}^{s}|v_{j}^{*}(x_{i})-w_{j}^{*}(x_{i})|^{2}\right)^{\frac{1}{2}}.\]
Then by (\ref{xdiffbound}), we have
\[\left(\mathbb{E}|X_{f,z}-X_{h,\widetilde{z}}|^{2}\right)^{\frac{1}{2}}\leq\rho_{T}(f,h)+\rho_{Z}(z_{1},z_{2}),\]
for all $(f,z),(h,\widetilde{z})\in T\times Z$. So by (\ref{gc}) and Lemma \ref{productgamma},
\begin{equation}\label{split}
\mathbb{E}\sup_{(f,z)\in T^{(b)}\times Z}X_{f,z}\leq C\gamma_{2}(T^{(b)}\times Z,\rho_{T}\times\rho_{Z})\leq3C\gamma_{2}(T^{(b)},\rho_{T})+3C\gamma_{2}(Z,\rho_{Z}).
\end{equation}
Let's bound each of these two terms. For the first term, by Lemma \ref{Dudleybound},
\begin{align}\label{tbbound}
&\gamma_{2}(T^{(b)},\rho_{T})\\\leq& C\int_{0}^{\infty}\sqrt{\log N(T^{(b)},\rho_{T},\epsilon)}\,d\epsilon\nonumber\\=&
C\int_{0}^{\infty}\sqrt{\log N(T^{(b)},\|\,\|_{(\delta)},\frac{\epsilon}{M})}\,d\epsilon\nonumber\\=&
CM\int_{0}^{\infty}\sqrt{\log N(T^{(b)},\|\,\|_{(\delta)},\epsilon)}\,d\epsilon\nonumber\\\leq&
CM\left(\int_{b}^{\infty}\sqrt{\log N(T^{(b)},\|\,\|_{(\delta)},\epsilon)}\,d\epsilon+b\sqrt{\log|T^{(b)}|}\right)\nonumber\\\leq&
CM\left(\int_{b}^{\infty}\sqrt{\log N(T,\|\,\|_{(\delta)},\frac{\epsilon}{2})}\,d\epsilon+b\sqrt{\log|T^{(b)}|}\right)\nonumber\\\leq&
CM\int_{\frac{b}{2}}^{\infty}\sqrt{\log N(T,\|\,\|_{(\delta)},\epsilon)}\,d\epsilon,\nonumber
\end{align}
where the second last inequality follows from $T^{(b)}\subset T$ and (\ref{coversubset}), and the last inequality follows from $|T^{(b)}|=N(T,\|\,\|_{(\delta)},b)$ (by definition of $T^{(b)}$) and $b=2\int_{\frac{b}{2}}^{b}1\,d\epsilon$.

We now bound the other term in (\ref{split}). Let $(g_{i,j})_{1\leq i\leq n,\,1\leq j\leq s}$ be i.i.d.~standard Gaussian random variables. Define the Gaussian process $(Y_{z})_{z\in Z}$ as follows. If $z=(v_{1}^{*},\ldots,v_{s}^{*})\in Z$ then
\[Y_{z}=\sum_{i=1}^{n}\sum_{j=1}^{s}g_{i,j}v_{j}^{*}(x_{i}).\]
Then for $z=(v_{1}^{*},\ldots,v_{s}^{*})\in Z$ and $\widetilde{z}=(w_{1}^{*},\ldots,w_{s}^{*})\in Z$, we have
\begin{eqnarray*}
(\mathbb{E}|Y_{z}-Y_{\widetilde{z}}|^{2})^{\frac{1}{2}}&=&\left(\mathbb{E}\left|\sum_{i=1}^{n}\sum_{j=1}^{s}g_{i,j}v_{j}^{*}(x_{i})-\sum_{i=1}^{n}\sum_{j=1}^{s}g_{i,j}w_{j}^{*}(x_{i})\right|^{2}\right)^{\frac{1}{2}}\\&=&
\left(\sum_{i=1}^{n}\sum_{j=1}^{s}|v_{j}^{*}(x_{i})-w_{j}^{*}(x_{i})|^{2}\right)^{\frac{1}{2}}\\&=&
\rho_{Z}((v_{1}^{*},\ldots,v_{s}^{*}),(w_{1}^{*},\ldots,w_{s}^{*}))=\rho(z,\widetilde{z}).
\end{eqnarray*}
So by Talagrand's majorizing measure theorem (\ref{gc}),
\begin{eqnarray*}
\gamma_{2}(Z,\rho_{Z})&\leq&C\cdot\mathbb{E}\sup_{z\in Z}Y_{z}\\&=&
C\cdot\mathbb{E}\sup_{(v_{1}^{*},\ldots,v_{s}^{*})\in Z}\sum_{i=1}^{n}\sum_{j=1}^{s}g_{i,j}v_{j}^{*}(x_{i})\\&=&
C\cdot\mathbb{E}\sum_{j=1}^{s}\sup_{v^{*}\in B_{E^{*}}}\sum_{i=1}^{n}g_{i,j}v^{*}(x_{i})\\&=&
C\cdot\mathbb{E}\sum_{j=1}^{s}\left\|\sum_{i=1}^{n}g_{i,j}x_{i}\right\|=Cs\cdot\mathbb{E}\left\|\sum_{i=1}^{n}g_{i}x_{i}\right\|.
\end{eqnarray*}
So we have bounded the second term in (\ref{split}). Together with the bound (\ref{tbbound}) for the first term, we obtain the following from (\ref{split}).
\[\mathbb{E}\sup_{(f,z)\in T^{(b)}\times Z}X_{f,z}\leq CM\int_{\frac{b}{2}}^{\infty}\sqrt{\log N(T,\|\,\|_{(\delta)},\epsilon)}\,d\epsilon+Cs\cdot\mathbb{E}\left\|\sum_{i=1}^{n}g_{i}x_{i}\right\|.\]
Combining this with (\ref{discretize}), we obtain
\begin{align*}
&\mathbb{E}\sup_{\substack{v_{1}^{*},\ldots,v_{s}^{*}\in B_{E^{*}}\\f\in T}}\left|\frac{1}{n}\sum_{i=1}^{n}g_{i}f(v_{1}^{*}(x_{i}),\ldots,v_{s}^{*}(x_{i}))\right|\\\leq&
C\inf_{b>0}\left(\frac{bM}{\sqrt{n}}+\frac{M}{n}\int_{b}^{\infty}\sqrt{\log N(T,\|\,\|_{(\delta)},\epsilon)}\,d\epsilon\right)+\frac{Cs}{n}\mathbb{E}\left\|\sum_{i=1}^{n}g_{i}x_{i}\right\|\\\leq&
C\inf_{0<b\leq 1}\left(\frac{bM}{\sqrt{n}}+\frac{M}{n}\int_{b}^{1}\sqrt{\left(\frac{C\sqrt{s}}{\epsilon}\right)^{s}\frac{1}{\delta}}\,d\epsilon\right)+\frac{Cs}{n}\mathbb{E}\left\|\sum_{i=1}^{n}g_{i}x_{i}\right\|\\&\text{by Lemma }\ref{fscovering}\\=&
\frac{Cs}{n}\mathbb{E}\left\|\sum_{i=1}^{n}g_{i}x_{i}\right\|+
\frac{CM}{\sqrt{n}}\cdot\inf_{0<b\leq 1}\left(b+\frac{1}{\sqrt{\delta n}}\int_{b}^{1}\sqrt{\left(\frac{C\sqrt{s}}{\epsilon}\right)^{s}}\,d\epsilon\right).
\end{align*}
To complete the proof of the result, it remains to show that
\[\inf_{0<b\leq 1}\left(b+\frac{1}{\sqrt{\delta n}}\int_{b}^{1}\sqrt{\left(\frac{C\sqrt{s}}{\epsilon}\right)^{s}}\,d\epsilon\right)\leq C\sqrt{s}\cdot\begin{cases}(\delta n)^{-\frac{1}{2}},&s=1\\(\ln(\delta n+2))\cdot(\delta n)^{-\frac{1}{2}},&s=2\\(\delta n)^{-\frac{1}{s}},&s\geq 3\end{cases}.\]
When $s=1$, this bound follows by taking $b=0$. When $s=2$, if $\delta n\geq 1$, take $b=(\delta n)^{-\frac{1}{2}}$; and if $\delta n<1$, take $b=1$ and the above bound follows. When $s=3$, if $C\sqrt{s}\cdot(\delta n)^{-\frac{1}{s}}\leq 1$, take $b=C\sqrt{s}\cdot(\delta n)^{-\frac{1}{s}}$; otherwise, take $b=1$ and the above bound follows.
\end{proof}
In the sequel, we define
\begin{equation}\label{Phidef}
\Phi(n,s,\delta)=\begin{cases}(\delta n)^{-\frac{1}{2}},&s=1\\(\ln(\delta n+2))\cdot(\delta n)^{-\frac{1}{2}},&s=2\\(\delta n)^{-\frac{1}{s}},&s\geq 3\end{cases}.
\end{equation}
\begin{lemma}\label{mainlemmascaled}
Let $0<\delta\leq 1$. Suppose that $E$ is a Banach space with separable dual $E^{*}$ and $x_{1},\ldots,x_{n}\in E$. Let $g_{1},\ldots,g_{n}$ be i.i.d.~standard Gaussian random variables. Let $T$ be the set of all $1$-Lipschitz functions $f:\mathbb{R}^{s}\to\mathbb{R}$ with $f(0)=0$. Then
\begin{align}\label{mainlemmascaledeq}
&\mathbb{E}\sup_{\substack{v_{1}^{*},\ldots,v_{s}^{*}\in B_{E^{*}}\\f\in T}}\left|\frac{1}{n}\sum_{i=1}^{n}g_{i}f(v_{1}^{*}(x_{i}),\ldots,v_{s}^{*}(x_{i}))\right|\\\leq&
\frac{Cs}{n}\mathbb{E}\left\|\sum_{i=1}^{n}g_{i}x_{i}\right\|+
Cs\cdot\left(\frac{1}{n}\sup_{v^{*}\in B_{E^{*}}}\sum_{i=1}^{n}|v^{*}(x_{i})|^{2+2\delta}\right)^{\frac{1}{2+2\delta}}\cdot\Phi(n,s,\delta).\nonumber
\end{align}
\end{lemma}
\begin{proof}
Observe that if $f\in T$ and $a>0$, then the map $y\mapsto \frac{1}{a}f(ay)$ from $\mathbb{R}^{s}$ to $\mathbb{R}$ is also in $T$. Thus, for every $a>0$, if we replace $x_{1},\ldots,x_{n}$ by $ax_{1},\ldots,ax_{n}$, then both sides of (\ref{mainlemmascaledeq}) will be multiplied by the same factor $a$ and thus it can be canceled out. So without loss of generality, by scaling $x_{1},\ldots,x_{n}$ by a suitably chosen $a>0$, we may assume that
\[\sup_{v^{*}\in B_{E^{*}}}\sum_{i=1}^{n}|v^{*}(x_{i})|^{2+2\delta}=n\cdot s^{-(1+\delta)}.\]
Then in Lemma \ref{mainlemma},
\[M\leq\sqrt{2}\left[\sqrt{n}+s^{\frac{1+\delta}{2}}\left(\sup_{v^{*}\in B_{E^{*}}}\sum_{i=1}^{n}|v^{*}(x_{i})|^{2+2\delta}\right)^{\frac{1}{2}}\right]=2\sqrt{2}\cdot\sqrt{n}.\]
So by Lemma \ref{mainlemma},
\begin{align*}
&\mathbb{E}\sup_{\substack{v_{1}^{*},\ldots,v_{s}^{*}\in B_{E^{*}}\\f\in T}}\left|\frac{1}{n}\sum_{i=1}^{n}g_{i}f(v_{1}^{*}(x_{i}),\ldots,v_{s}^{*}(x_{i}))\right|\\\leq&
\frac{Cs}{n}\mathbb{E}\left\|\sum_{i=1}^{n}g_{i}x_{i}\right\|+C\sqrt{s}\cdot\Phi(n,s,\delta)\\=&
\frac{Cs}{n}\mathbb{E}\left\|\sum_{i=1}^{n}g_{i}x_{i}\right\|+Cs\cdot\left(\frac{1}{n}\sup_{v^{*}\in B_{E^{*}}}\sum_{i=1}^{n}|v^{*}(x_{i})|^{2+2\delta}\right)^{\frac{1}{2+2\delta}}\cdot\Phi(n,s,\delta),
\end{align*}
since we assume that $\displaystyle\sup_{v^{*}\in B_{E^{*}}}\sum_{i=1}^{n}|v^{*}(x_{i})|^{2+2\delta}=n\cdot s^{-(1+\delta)}$. So the result follows.
\end{proof}
\begin{theorem}\label{mainsampling}
Let $0<\delta\leq 1$. Suppose that $\mu$ is a probability measure on a Banach space $E$ with separable dual $E^{*}$ and $\int_{E}\|x\|\,d\mu(x)<\infty$. Let $X_{1},\ldots,X_{n}$ be i.i.d.~random points in $E$ sampled according to $\mu$. Then
\begin{align*}
&\mathbb{E}W_{1,s}\left(\mu,\frac{1}{n}\sum_{i=1}^{n}\delta_{X_{i}}\right)\\\leq&
\frac{Cs}{n}\mathbb{E}\left\|\sum_{i=1}^{n}g_{i}X_{i}\right\|+
Cs\cdot\mathbb{E}\left[\left(\frac{1}{n}\sup_{v^{*}\in B_{E^{*}}}\sum_{i=1}^{n}|v^{*}(X_{i})|^{2+2\delta}\right)^{\frac{1}{2+2\delta}}\right]\cdot\Phi(n,s,\delta),
\end{align*}
where $g_{1},\ldots,g_{n}$ are i.i.d.~standard Gaussian random variables that are independent of $X_{1},\ldots,X_{n}$, and $\Phi(n,s,\delta)$ is defined in (\ref{Phidef}).
\end{theorem}
\begin{proof}
By the definition of $W_{1,s}$ in Section \ref{defsection},
\begin{align*}
&W_{1,s}\left(\mu,\frac{1}{n}\sum_{i=1}^{n}\delta_{X_{i}}\right)\\=&
\sup_{\substack{v_{1}^{*},\ldots,v_{s}^{*}\in B_{E^{*}}\\f\text{ is 1-Lipschitz}}}\left|\frac{1}{n}\sum_{i=1}^{n}f(v_{1}^{*}(X_{i}),\ldots,v_{s}^{*}(X_{i}))-\int_{E}f(v_{1}^{*}(x),\ldots,v_{s}^{*}(x))\,d\mu(x)\right|\\=&
\sup_{F\in\mathcal{F}}\left(\frac{1}{n}\sum_{i=1}^{n}F(X_{i})-\int F(x)\,d\mu(x)\right),
\end{align*}
where $\mathcal{F}$ consists of all functions $F:E\to\mathbb{R}$ of the form $F(x)=f(v_{1}^{*}(x),\ldots,v_{s}^{*}(x))$ for some $v_{1}^{*},\ldots,v_{s}^{*}\in B_{E^{*}}$ and $1$-Lipschitz functions $f:\mathbb{R}^{s}\to\mathbb{R}$ with $f(0)=0$. But by Gaussian symmetrization (see \cite[Lemma 7.4]{vHnotes}),
\[\mathbb{E}\sup_{F\in\mathcal{F}}\left(\frac{1}{n}\sum_{i=1}^{n}F(X_{i})-\int F(x)\,d\mu(x)\right)\leq\sqrt{2\pi}\cdot\mathbb{E}\sup_{F\in\mathcal{F}}\frac{1}{n}\sum_{i=1}^{n}g_{i}F(X_{i}).\]
Therefore,
\[\mathbb{E}W_{1,s}\left(\mu,\frac{1}{n}\sum_{i=1}^{n}\delta_{X_{i}}\right)\leq
\sqrt{2\pi}\cdot\mathbb{E}\sup_{\substack{v_{1}^{*},\ldots,v_{s}^{*}\in B_{E^{*}}\\f\text{ is 1-Lipschitz}\\f(0)=0}}\frac{1}{n}\sum_{i=1}^{n}g_{i}f(v_{1}^{*}(X_{i}),\ldots,v_{s}^{*}(X_{i})).\]
So by Lemma \ref{mainlemmascaled}, the result follows.
\end{proof}
\begin{corollary}\label{mainsamplinghilbert}
Let $0<\delta\leq 1$. Suppose that $\mu$ is a probability measure on a separable Hilbert space $E$ with $\int_{E}\|x\|^{2+2\delta}\,d\mu(x)<\infty$. Let $X_{1},\ldots,X_{n}$ be i.i.d.~random points in $E$ sampled according to $\mu$. Then
\[\mathbb{E}W_{1,s}\left(\mu,\frac{1}{n}\sum_{i=1}^{n}\delta_{X_{i}}\right)\leq Cs\cdot\left(\int_{E}\|x\|^{2+2\delta}\,d\mu(x)\right)^{\frac{1}{2+2\delta}}\cdot\Phi(n,s,\delta),\]
where $\Phi(n,s,\delta)$ is defined in (\ref{Phidef}).
\end{corollary}
\begin{proof}
In Theorem \ref{mainsampling},
\[\mathbb{E}\left\|\sum_{i=1}^{n}g_{i}X_{i}\right\|\leq
\left(\mathbb{E}\left\|\sum_{i=1}^{n}g_{i}X_{i}\right\|^{2}\right)^{\frac{1}{2}}=
\left(\sum_{i=1}^{n}\mathbb{E}\|X_{i}\|^{2}\right)^{\frac{1}{2}}=\sqrt{n}\left(\int_{E}\|x\|^{2}\,d\mu(x)\right)^{\frac{1}{2}}.\]
We also have
\begin{eqnarray*}
\mathbb{E}\left[\left(\frac{1}{n}\sup_{v^{*}\in B_{E^{*}}}\sum_{i=1}^{n}|v^{*}(X_{i})|^{2+2\delta}\right)^{\frac{1}{2+2\delta}}\right]&\leq&\mathbb{E}\left[\left(\frac{1}{n}\sum_{i=1}^{n}\|X_{i}\|^{2+2\delta}\right)^{\frac{1}{2+2\delta}}\right]\\&\leq&
\left(\int_{E}\|x\|^{2+2\delta}\,d\mu(x)\right)^{\frac{1}{2+2\delta}}.
\end{eqnarray*}
Since $\frac{1}{\sqrt{n}}\leq\Phi(s,\delta,n)$, by Theorem \ref{mainsampling}, the result follows.
\end{proof}
\begin{corollary}\label{mainsamplingbanach}
Suppose that $\mu$ is a probability measure on a Banach space $E$ with separable dual $E^{*}$ and $\int_{E}\|x\|\,d\mu(x)<\infty$. Let $X_{1},\ldots,X_{n}$ be i.i.d.~random points in $E$ sampled according to $\mu$. Then
\[\mathbb{E}W_{1,1}\left(\mu,\frac{1}{n}\sum_{i=1}^{n}\delta_{X_{i}}\right)\leq
\frac{C}{n}\mathbb{E}\left\|\sum_{i=1}^{n}g_{i}X_{i}\right\|+\frac{C\sqrt{\ln n}}{n}\cdot\mathbb{E}\sup_{v^{*}\in B_{E^{*}}}\left(\sum_{i=1}^{n}|v^{*}(X_{i})|^{2}\right)^{\frac{1}{2}},\]
where $g_{1},\ldots,g_{n}$ are i.i.d.~standard Gaussian random variables that are independent of $X_{1},\ldots,X_{n}$.
\end{corollary}
\begin{proof}
By Theorem \ref{mainsampling} with $s=1$,
\begin{align*}
&\mathbb{E}W_{1,1}\left(\mu,\frac{1}{n}\sum_{i=1}^{n}\delta_{X_{i}}\right)\\\leq&
\frac{C}{n}\mathbb{E}\left\|\sum_{i=1}^{n}g_{i}X_{i}\right\|+\frac{C}{\sqrt{n}}\cdot\inf_{0<\delta\leq 1}\frac{1}{\sqrt{\delta}}\mathbb{E}\left[\left(\frac{1}{n}\sup_{v^{*}\in B_{E^{*}}}\sum_{i=1}^{n}|v^{*}(X_{i})|^{2+2\delta}\right)^{\frac{1}{2+2\delta}}\right]\\\leq&
\frac{C}{n}\mathbb{E}\left\|\sum_{i=1}^{n}g_{i}X_{i}\right\|+\frac{C}{\sqrt{n}}\cdot\inf_{0<\delta\leq 1}\frac{1}{\sqrt{\delta}}\mathbb{E}\left[n^{-\frac{1}{2+2\delta}}\left(\sup_{v^{*}\in B_{E^{*}}}\sum_{i=1}^{n}|v^{*}(X_{i})|^{2}\right)^{\frac{1}{2}}\right].
\end{align*}
Take $\delta=1/\lceil\ln n\rceil$. Then $n^{-\frac{1}{2+2\delta}}\leq\frac{C}{\sqrt{n}}$. The result follows.
\end{proof}
In the rest of this section, we prove some lower bound results. These results are quite standard.
\begin{proposition}\label{lbthm}
Suppose that $\mu$ is a probability measure on a Banach space $E$ with separable dual $E^{*}$ and that $\int_{E}\|x\|\,d\mu(x)<\infty$ and $\int_{E}x\,d\mu(x)=0$. Let $X_{1},\ldots,X_{n}$ be i.i.d.~random points in $E$ sampled according to $\mu$. Then
\[\mathbb{E}W_{1,1}\left(\mu,\frac{1}{n}\sum_{i=1}^{n}\delta_{X_{i}}\right)\geq\frac{1}{2n}\mathbb{E}\left\|\sum_{i=1}^{n}\epsilon_{i}X_{i}\right\|,\]
where $\epsilon_{1},\ldots,\epsilon_{n}$ are i.i.d.~uniform $\pm 1$ random variables that are independent of $X_{1},\ldots,X_{n}$.
\end{proposition}
\begin{proof}
For fixed $x_{1},\ldots,x_{n}\in E$, by considering the 1-Lipschitz function $f(t)=t$, we have
\[W_{1,1}\left(\mu,\frac{1}{n}\sum_{i=1}^{n}\delta_{x_{i}}\right)\geq\sup_{v^{*}\in B_{E^{*}}}\left|\int_{E}v^{*}(x)\,d\mu(x)-\frac{1}{n}\sum_{i=1}^{n}v^{*}(x_{i})\right|=\left\|\frac{1}{n}\sum_{i=1}^{n}x_{i}\right\|.\]
So
\[\mathbb{E}W_{1,1}\left(\mu,\frac{1}{n}\sum_{i=1}^{n}\delta_{X_{i}}\right)\geq\mathbb{E}\left\|\frac{1}{n}\sum_{i=1}^{n}X_{i}\right\|.\]
Let $Y_{1},\ldots,Y_{n}$ be i.i.d.~random points in $E$ sampled according to $\mu$ that are independent of $X_{1},\ldots,X_{n}$ and $\epsilon_{1},\ldots,\epsilon_{n}$. Then
\[\mathbb{E}\left\|\sum_{i=1}^{n}X_{i}\right\|\geq\frac{1}{2}\mathbb{E}\left\|\sum_{i=1}^{n}(X_{i}-Y_{i})\right\|=
\frac{1}{2}\mathbb{E}\left\|\sum_{i=1}^{n}\epsilon_{i}(X_{i}-Y_{i})\right\|\geq\frac{1}{2}\mathbb{E}\left\|\sum_{i=1}^{n}\epsilon_{i}X_{i}\right\|,\]
where the last inequality follows from Jensen's inequality and taking expectation on $Y_{1},\ldots,Y_{n}$. The result follows.
\end{proof}
\begin{corollary}\label{lbthmhilbert}
Suppose that $\mu$ is a probability measure on a separable Hilbert space $E$ with $\int_{E}\|x\|\,d\mu(x)<\infty$ and $\int_{E}x\,d\mu(x)=0$. Let $X_{1},\ldots,X_{n}$ be i.i.d.~random points in $E$ sampled according to $\mu$. Then
\[\mathbb{E}W_{1,1}\left(\mu,\frac{1}{n}\sum_{i=1}^{n}\delta_{X_{i}}\right)\geq\frac{1}{2\sqrt{2n}}\int_{E}\|x\|\,d\mu(x).\]
\end{corollary}
\begin{proof}
By Proposition \ref{lbthm}, it suffices to show that
\[\mathbb{E}\left\|\sum_{i=1}^{n}\epsilon_{i}X_{i}\right\|\geq\sqrt{\frac{n}{2}}\cdot\mathbb{E}\|X_{1}\|.\]
If we first take expectation on $\epsilon_{1},\ldots,\epsilon_{n}$, then by the Kahane-Khintchine inequality \cite{Latala}, we have
\[\mathbb{E}_{\epsilon}\left\|\sum_{i=1}^{n}\epsilon_{i}X_{i}\right\|\geq\frac{1}{\sqrt{2}}\left(\mathbb{E}_{\epsilon}\left\|\sum_{i=1}^{n}\epsilon_{i}X_{i}\right\|^{2}\right)^{\frac{1}{2}}=\frac{1}{\sqrt{2}}\left(\sum_{i=1}^{n}\|X_{i}\|^{2}\right)^{\frac{1}{2}}\geq\frac{1}{\sqrt{2n}}\sum_{i=1}^{n}\|X_{i}\|.\]
So
\[\mathbb{E}\left\|\sum_{i=1}^{n}\epsilon_{i}X_{i}\right\|\geq\frac{1}{\sqrt{2n}}\sum_{i=1}^{n}\mathbb{E}\|X_{i}\|=\sqrt{\frac{n}{2}}\cdot\mathbb{E}\|X_{1}\|.\]
\end{proof}
\section{Max-sliced 2-Wasserstein distance}
The following lemma is known. See e.g., \cite{Rudelson}.
\begin{lemma}\label{rmbound}
Let $r>0$. Suppose that $\mu$ is a probability measure on $\mathbb{R}^{d}$ supported on $\{x\in\mathbb{R}^{d}:\,\|x\|_{2}\leq r\}$. Let $X_{1},\ldots,X_{n}$ be i.i.d.~random vectors in $\mathbb{R}^{d}$ sampled according to $\mu$. Let $g_{1},\ldots,g_{n}$ be i.i.d.~standard Gaussian random variables that are independent of $X_{1},\ldots,X_{n}$. Then
\[\mathbb{E}\left\|\sum_{i=1}^{n}X_{i}X_{i}^{T}\right\|_{\mathrm{op}}\leq 2n\|\mathbb{E}X_{1}X_{1}^{T}\|_{\mathrm{op}}+Cr^{2}\ln n,\]
and
\[\mathbb{E}\left\|\sum_{i=1}^{n}g_{i}X_{i}X_{i}^{T}\right\|_{\mathrm{op}}\leq Cr\sqrt{n\ln n}\,\|\mathbb{E}X_{1}X_{1}^{T}\|_{\mathrm{op}}^{\frac{1}{2}}+Cr^{2}\ln n.\]
\end{lemma}
\begin{proof}
Fix $x_{1},\ldots,x_{n}\in\mathbb{R}^{d}$ with $\|x_{i}\|_{2}\leq r$ for all $i$. By the noncommutative Khintchine inequality (see, e.g., \cite[Proposition 2.3]{Tropp2ndorder}), for $p\in\mathbb{N}$, one can bound the following expected trace 
\begin{eqnarray*}
\mathbb{E}\,\mathrm{Tr}\left(\sum_{i=1}^{n}g_{i}x_{i}x_{i}^{T}\right)^{2p}&\leq&
(\sqrt{2p})^{2p}\,\mathrm{Tr}\left[\left(\sum_{i=1}^{n}(x_{i}x_{i}^{T})^{2}\right)^{p}\,\right]\\&=&
(\sqrt{2p})^{2p}\,\mathrm{Tr}\left[\left(\sum_{i=1}^{n}\|x_{i}\|_{2}^{2}\,x_{i}x_{i}^{T}\right)^{p}\,\right]\\&\leq&
(\sqrt{2p})^{2p}n\,\left\|\left(\sum_{i=1}^{n}\|x_{i}\|_{2}^{2}\,x_{i}x_{i}^{T}\right)^{p}\,\right\|_{\mathrm{op}}\\&\leq&
(r\sqrt{2p})^{2p}n\left\|\sum_{i=1}^{n}x_{i}x_{i}^{T}\right\|_{\mathrm{op}}^{p},
\end{eqnarray*}
where the second last inequality follows from the fact that $\sum_{i=1}^{n}\|x_{i}\|_{2}^{2}\,x_{i}x_{i}^{T}$ has rank at most $n$. Since the trace of any positive semidefinite matrix is greater than or equal to its operator norm, it follows that
\[\mathbb{E}\left\|\sum_{i=1}^{n}g_{i}x_{i}x_{i}^{T}\right\|_{\mathrm{op}}\leq r\sqrt{2p}\cdot n^{\frac{1}{2p}}\left\|\sum_{i=1}^{n}x_{i}x_{i}^{T}\right\|_{\mathrm{op}}^{\frac{1}{2}}.\]
Taking $p=\lceil\ln n\rceil$, we obtain
\[\mathbb{E}\left\|\sum_{i=1}^{n}g_{i}x_{i}x_{i}^{T}\right\|_{\mathrm{op}}\leq Cr\sqrt{\ln n}\left\|\sum_{i=1}^{n}x_{i}x_{i}^{T}\right\|_{\mathrm{op}}^{\frac{1}{2}}.\]
Now we randomize $x_{1},\ldots,x_{n}$. We get
\begin{equation}\label{rmboundeq}
\mathbb{E}\left\|\sum_{i=1}^{n}g_{i}X_{i}X_{i}^{T}\right\|_{\mathrm{op}}\leq Cr\sqrt{\ln n}\left(\mathbb{E}\left\|\sum_{i=1}^{n}X_{i}X_{i}^{T}\right\|_{\mathrm{op}}\right)^{\frac{1}{2}}.
\end{equation}
By symmetrization (see \cite[Lemma 6.4.2]{Romanbook} and \cite[Exercise 7.1]{vHnotes}),
\[\mathbb{E}\left\|\sum_{i=1}^{n}(X_{i}X_{i}^{T}-\mathbb{E}X_{i}X_{i}^{T})\right\|_{\mathrm{op}}\leq C\cdot\mathbb{E}\left\|\sum_{i=1}^{n}g_{i}X_{i}\right\|.\]
Thus,
\begin{eqnarray*}
\mathbb{E}\left\|\sum_{i=1}^{n}X_{i}X_{i}^{T}\right\|_{\mathrm{op}}&\leq&\|n\mathbb{E}X_{1}X_{1}^{T}\|_{\mathrm{op}}+C\cdot\mathbb{E}\left\|\sum_{i=1}^{n}g_{i}X_{i}X_{i}^{T}\right\|_{\mathrm{op}}\\&\leq&
n\|\mathbb{E}X_{1}X_{1}^{T}\|_{\mathrm{op}}+Cr\sqrt{\ln n}\left(\mathbb{E}\left\|\sum_{i=1}^{n}X_{i}X_{i}^{T}\right\|_{\mathrm{op}}\right)^{\frac{1}{2}},
\end{eqnarray*}
where the last inequality follows from (\ref{rmboundeq}). So
\[\mathbb{E}\left\|\sum_{i=1}^{n}X_{i}X_{i}^{T}\right\|_{\mathrm{op}}\leq 2n\|\mathbb{E}X_{1}X_{1}^{T}\|_{\mathrm{op}}+Cr^{2}\ln n.\]
This proves the first inequality. Combining this with (\ref{rmboundeq}), we obtain the second inequality.
\end{proof}
\begin{lemma}\label{2to1}
Suppose that $\mu,\widetilde{\mu}$ are symmetric probability measures on $(\mathbb{R}^{d},\|\,\|_{2})$ supported on $\{x\in\mathbb{R}^{d}:\,\|x\|_{2}\leq r\}$. Consider the map $\Gamma(x):=xx^{T}$ from the Hilbert space $(\mathbb{R}^{d},\|\,\|_{2})$ to the Banach space $(\mathbb{R}^{d\times d},\|\,\|_{\mathrm{op}})$. Let $\Gamma_{\#}\mu$ and $\Gamma_{\#}\widetilde{\mu}$ be the pushforward measures of $\mu$ and $\widetilde{\mu}$ by $\Gamma$, respectively. Then
\[W_{2,1}(\mu,\widetilde{\mu})^{2}\leq W_{1,1}(\Gamma_{\#}\mu,\Gamma_{\#}\widetilde{\mu}).\]
\end{lemma}
\begin{proof}
Define $\mathrm{abs}:\mathbb{R}\to\mathbb{R}$ 
and $\mathrm{sq}:\mathbb{R}\to\mathbb{R}$ by $\mathrm{abs}(t)=|t|$ and $\mathrm{sq}(t)=t^{2}$. Observe that if $\nu$ and $\widetilde{\nu}$ are symmetric probability measures on the interval $[-r,r]$, then
\begin{eqnarray*}
W_{2}(\nu,\widetilde{\nu})^{2}&=&W_{2}(\mathrm{abs}_{\#}\nu,\mathrm{abs}_{\#}\widetilde{\nu})^{2}\\&=&\inf_{\gamma}\int_{[0,r]\times[0,r]}|t-s|^{2}\,d\gamma(t,s)\\&\leq&
\inf_{\gamma}\int_{[0,r]\times[0,r]}|t^{2}-s^{2}|\,d\gamma(t,s)\\&=&
W_{1}(\mathrm{sq}_{\#}\mathrm{abs}_{\#}\nu,\mathrm{sq}_{\#}\mathrm{abs}_{\#}\widetilde{\nu})\\&=&
W_{1}(\mathrm{sq}_{\#}\nu,\mathrm{sq}_{\#}\widetilde{\nu}),
\end{eqnarray*}
where the infimum is over all coupling $\gamma$ of the pushforward measures $\mathrm{abs}_{\#}\nu$ and $\mathrm{abs}_{\#}\widetilde{\nu}$ on $[0,r]$.

For $u\in\mathbb{R}^{d}$ with $\|u\|_{2}=1$, let $u_{\#}\mu$ and $u_{\#}\widetilde{\mu}$ be the pushforward measure of $\mu$ and $\widetilde{\mu}$, respectively, by the map $\langle\cdot,u\rangle$. Taking $\nu=u_{\#}\mu$ and $\widetilde{\nu}=u_{\#}\widetilde{\mu}$ in the above, we obtain
\begin{eqnarray*}
W_{2,1}(\mu,\widetilde{\mu})^{2}&=&\sup_{u\in\mathbb{R}^{d},\,\|u\|_{2}=1}W_{2}(u_{\#}\mu,\,u_{\#}\widetilde{\mu})^{2}\\&\leq&
\sup_{u\in\mathbb{R}^{d},\,\|u\|_{2}=1}W_{1}(\mathrm{sq}_{\#}u_{\#}\mu,\,\mathrm{sq}_{\#}u_{\#}\widetilde{\mu})
\end{eqnarray*}
Observe that $\mathrm{sq}_{\#}u_{\#}\mu$ is the pushforward measure of $\mu$ by the map
\[x\mapsto\langle x,u\rangle^{2}=\mathrm{Tr}(uu^{T}xx^{T})=\mathrm{Tr}(uu^{T}\Gamma(x)),\]
and similarly, $\mathrm{sq}_{\#}u_{\#}\widetilde{\mu}$ is the pushforward measure of $\widetilde{\mu}$ by the same map. Moreover, since the trace class norm of $uu^{T}$ is equal to $1$, we can identify $uu^{T}$ as an element of the unit ball of the dual of the Banach space $(\mathbb{R}^{d\times d},\|\,\|_{\mathrm{op}})$. Thus the result follows.
\end{proof}
Next we prove Theorem \ref{mainsamplingone} in the introduction, where the statement is copied below for the reader's convenience.
\begin{thmmainsamplingone}%\label{mainsamplingonerestate}
Let $r>0$. Suppose that $\mu$ is a symmetric probability measure on $\mathbb{R}^{d}$ supported on $\{x\in\mathbb{R}^{d}:\,\|x\|_{2}\leq r\}$. Let $X_{1},\ldots,X_{n}$ be i.i.d.~random vectors sampled according to $\mu$. Then
\[\mathbb{E}\left[W_{2,1}\left(\mu,\frac{1}{2n}\sum_{i=1}^{n}(\delta_{X_{i}}+\delta_{-X_{i}})\right)^{2}\right]\leq
C\|\Sigma\|_{\mathrm{op}}\left(\frac{r^{2}\ln n}{n\|\Sigma\|_{\mathrm{op}}}+\sqrt{\frac{r^{2}\ln n}{n\|\Sigma\|_{\mathrm{op}}}}\,\right),\]
where $\displaystyle\Sigma=\int_{\mathbb{R}^{d}}xx^{T}\,d\mu(x)$.
\end{thmmainsamplingone}
\begin{proof}
Since $\mu$ and $\frac{1}{2n}\sum_{i=1}^{n}(\delta_{X_{i}}+\delta_{-X_{i}})$ are symmetric, by Lemma \ref{2to1},
\begin{eqnarray}\label{2to1applied}
W_{2,1}\left(\mu,\frac{1}{2n}\sum_{i=1}^{n}(\delta_{X_{i}}+\delta_{-X_{i}})\right)^{2}&\leq& W_{1,1}\left(\Gamma_{\#}\mu,\,\Gamma_{\#}\left[\frac{1}{2n}\sum_{i=1}^{n}(\delta_{X_{i}}+\delta_{-X_{i}})\right]\right)\\&=&
W_{1,1}\left(\Gamma_{\#}\mu,\,\frac{1}{n}\sum_{i=1}^{n}\delta_{\Gamma(X_{i})}\right).\nonumber
\end{eqnarray}
Note that $\Gamma(X_{i})=X_{i}X_{i}^{T}$ are i.i.d.~random matrices with distribution $\Gamma_{\#}\mu$. Taking $E=(\mathbb{R}^{d\times d},\|\,\|_{\mathrm{op}})$ in Corollary \ref{mainsamplingbanach}, we obtain
\begin{align*}
&\mathbb{E}W_{1,1}\left(\Gamma_{\#}\mu,\,\frac{1}{n}\sum_{i=1}^{n}\delta_{\Gamma(X_{i})}\right)\\\leq&
\frac{C}{n}\mathbb{E}\left\|\sum_{i=1}^{n}g_{i}X_{i}X_{i}^{T}\right\|+\frac{C\sqrt{\ln n}}{n}\cdot\mathbb{E}\sup_{V^{*}\in B_{E^{*}}}\left(\sum_{i=1}^{n}|V^{*}(X_{i}X_{i}^{T})|^{2}\right)^{\frac{1}{2}}
\end{align*}
The dual ball $B_{E^{*}}$ consists of all $d\times d$ matrices with trace class norm at most $1$. This coincides with the convex hull of $\{vw^{T}:\,v,w\in\mathbb{R}^{d},\,\|v\|_{2}\leq 1,\,\|w\|_{2}\leq 1\}$ by the singular value decomposition. So
\begin{eqnarray*}
\sup_{V^{*}\in B_{E^{*}}}\left(\sum_{i=1}^{n}|V^{*}(X_{i}X_{i}^{T})|^{2}\right)^{\frac{1}{2}}&=&
\sup_{\substack{v,w\in\mathbb{R}^{d}\\\|v\|_{2},\|w\|_{2}\leq 1}}\left(\sum_{i=1}^{n}|\mathrm{Tr}(vw^{T}X_{i}X_{i}^{T})|^{2}\right)^{\frac{1}{2}}\\&=&
\sup_{\substack{v,w\in\mathbb{R}^{d}\\\|v\|_{2},\|w\|_{2}\leq 1}}\left(\sum_{i=1}^{n}\langle X_{i},v\rangle^{2}\langle X_{i},w\rangle^{2}\right)^{\frac{1}{2}}\\&\leq&
r\sup_{v\in\mathbb{R}^{d},\,\|v\|_{2}=1}\left(\sum_{i=1}^{n}\langle X_{i},v\rangle^{2}\right)^{\frac{1}{2}}=r\left\|\sum_{i=1}^{n}X_{i}X_{i}^{T}\right\|_{\mathrm{op}}^{\frac{1}{2}}.
\end{eqnarray*}
Therefore,
\[\mathbb{E}W_{1,1}\left(\Gamma_{\#}\mu,\,\frac{1}{n}\sum_{i=1}^{n}\delta_{\Gamma(X_{i})}\right)\leq\frac{C}{n}\mathbb{E}\left\|\sum_{i=1}^{n}g_{i}X_{i}X_{i}^{T}\right\|+\frac{Cr\sqrt{\ln n}}{n}\cdot\mathbb{E}\left\|\sum_{i=1}^{n}X_{i}X_{i}^{T}\right\|_{\mathrm{op}}^{\frac{1}{2}}.\]
So by Lemma \ref{rmbound},
\begin{align*}
&\mathbb{E}W_{1,1}\left(\Gamma_{\#}\mu,\,\frac{1}{n}\sum_{i=1}^{n}\delta_{\Gamma(X_{i})}\right)\\\leq&
\frac{Cr\sqrt{\ln n}}{\sqrt{n}}\|\mathbb{E}X_{1}X_{1}^{T}\|_{\mathrm{op}}^{\frac{1}{2}}+\frac{Cr^{2}\ln n}{n}+\frac{Cr\sqrt{\ln n}}{\sqrt{n}}\|\mathbb{E}X_{1}X_{1}^{T}\|_{\mathrm{op}}^{\frac{1}{2}}+\frac{Cr^{2}\ln n}{n}\\=&
\frac{Cr\sqrt{\ln n}}{\sqrt{n}}\|\Sigma\|_{\mathrm{op}}^{\frac{1}{2}}+\frac{Cr^{2}\ln n}{n}.
\end{align*}
So by (\ref{2to1applied}), the result follows.
\end{proof}
Finally, we prove Proposition \ref{mainsamplinglbone} in the introduction, where the statement is copied below.
\begin{propmainsamplinglbone}%\label{mainsamplinglbonerestate}
Let $\Sigma$ be a $d\times d$ positive semidefinite matrix such that $\|\Sigma\|_{\mathrm{op}}\leq\frac{1}{2}\mathrm{Tr}(\Sigma)$. Then there exists a symmetric probability measure $\mu$ on $\mathbb{R}^{d}$ supported on $\{x\in\mathbb{R}^{d}:\,\|x\|_{2}^{2}=\mathrm{Tr}(\Sigma)\}$ such that $\int_{\mathbb{R}^{d}}xx^{T}\,d\mu(x)=\Sigma$ and for every $n\in\mathbb{N}$, 
\begin{equation}\label{mainsamplinglbonerestateeq}
\mathbb{E}\left[W_{2,1}\left(\mu,\frac{1}{2n}\sum_{i=1}^{n}(\delta_{X_{i}}+\delta_{-X_{i}})\right)^{2}\right]\geq
\frac{1}{16}\|\Sigma\|_{\mathrm{op}}\left(\frac{\mathrm{Tr}(\Sigma)}{n\|\Sigma\|_{\mathrm{op}}}+\sqrt{\frac{\mathrm{Tr}(\Sigma)}{n\|\Sigma\|_{\mathrm{op}}}}\,\right),
\end{equation}
where $X_{1},\ldots,X_{n}$ are i.i.d.~random vectors sampled according to $\mu$.
\end{propmainsamplinglbone}
\begin{proof}
Without loss of generality, we may assume that $\Sigma$ is a diagonal matrix with diagonal entries $\lambda_{1}\geq\ldots\geq\lambda_{d}\geq 0$. We may also assume that $\mathrm{Tr}(\Sigma)=\lambda_{1}+\ldots+\lambda_{d}=1$. Let $\{e_{1},\ldots,e_{d}\}$ be the unit vector basis for $\mathbb{R}^{d}$. Take
\[\mu(\{e_{j}\})=\mu(\{-e_{j}\})=\frac{1}{2}\lambda_{j}\quad\text{for }j=1,\ldots,d.\]
Then $\mu$ is symmetric.

We need to show that the left hand side of (\ref{mainsamplinglbonerestateeq}) is at least each of the two terms on the right hand side. So the proof has two parts. The first part of the proof is similar to the proofs of Proposition \ref{lbthm} and Corollary \ref{lbthmhilbert}. Since the max-sliced $2$-Wasserstein distance $W_{2,1}$ is at least the max-sliced $1$-Wasserstein $W_{1,1}$, it suffices to bound $W_{1,1}$ from below. Let $\mathbb{R}_{+}^{d}=\{(v_{1},\ldots,v_{d}):\,v_{1},\ldots,v_{d}\geq 0\}$. By considering the 1-Lipschitz function $f(t)=|t|$, we have
\begin{align*}
&W_{1,1}\left(\mu,\frac{1}{2n}\sum_{i=1}^{n}(\delta_{X_{i}}+\delta_{-X_{i}})\right)\\\geq&
\sup_{v\in\mathbb{R}_{+}^{d},\,\|v\|_{2}=1}\left|\int_{\mathbb{R}^{d}}|\langle x,v\rangle|\,d\mu(x)-\frac{1}{2n}\sum_{i=1}^{n}(|\langle X_{i},v\rangle|+|\langle-X_{i},v\rangle|)\right|\\=&
\sup_{v\in\mathbb{R}_{+}^{d},\,\|v\|_{2}=1}\left|\sum_{i=1}^{d}\lambda_{i}v_{i}-\frac{1}{n}\sum_{i=1}^{n}|\langle X_{i},v\rangle|\right|\\=&
\sup_{v\in\mathbb{R}_{+}^{d},\,\|v\|_{2}=1}\left|\sum_{i=1}^{d}\lambda_{i}v_{i}-\frac{1}{n}\sum_{i=1}^{n}\langle \mathrm{abs}(X_{i}),v\rangle\right|,
\end{align*}
where $\mathrm{abs}(X_{i})$ is the vector for which we take absolute value on each entry of $X_{i}$. (Since $X_{i}$ is distributed according to $\mu$, the vector $X_{i}$ actually has only one nonzero entry.) So
\[W_{1,1}\left(\mu,\frac{1}{2n}\sum_{i=1}^{n}(\delta_{X_{i}}+\delta_{-X_{i}})\right)\geq\frac{1}{2}\left\|\mathrm{diag}(\Sigma)-\frac{1}{n}\sum_{i=1}^{n}\mathrm{abs}(X_{i})\right\|_{2},\]
where $\mathrm{diag}(\Sigma)=(\lambda_{1},\ldots,\lambda_{d})\in\mathbb{R}^{d}$.

Since $X_{1},\ldots,X_{n}$ are i.i.d.~with distribution $\mu$, the random vectors $\mathrm{abs}(X_{1}),\ldots,\mathrm{abs}(X_{n})$ are i.i.d.~with the following distribution
\[\mathrm{abs}_{\#}\mu(\{e_{j}\})=\lambda_{j}\quad\text{for }j=1,\ldots,d.\]
In particular, $\displaystyle\mathbb{E}[\mathrm{abs}(X_{1})]=\sum_{j=1}^{d}\lambda_{j}e_{j}=\mathrm{diag}(\Sigma)$. So
\begin{eqnarray*}
\mathbb{E}\left[W_{1,1}\left(\mu,\frac{1}{2n}\sum_{i=1}^{n}(\delta_{X_{i}}+\delta_{-X_{i}})\right)^{2}\right]&\geq&
\frac{1}{4}\mathbb{E}\left\|\mathrm{diag}(\Sigma)-\frac{1}{n}\sum_{i=1}^{n}\mathrm{abs}(X_{i})\right\|_{2}^{2}\\&=&
\frac{1}{4}\mathbb{E}\left\|\frac{1}{n}\sum_{i=1}^{n}(\mathrm{abs}(X_{i})-\mathbb{E}[\mathrm{abs}(X_{i})])\right\|_{2}^{2}
\\&=&
\frac{1}{4n}\mathbb{E}\left\|\mathrm{abs}(X_{1})-\mathbb{E}[\mathrm{abs}(X_{1})]\right\|_{2}^{2}\\&=&
\frac{1}{4n}\left[\mathbb{E}\|\mathrm{abs}(X_{1})\|_{2}^{2}-\|\mathbb{E}[\mathrm{abs}(X_{1})]\|_{2}^{2}\right]\\&=&
\frac{1}{4n}\left[\mathbb{E}\|\mathrm{abs}(X_{1})\|_{2}^{2}-\|\mathrm{diag}(\Sigma)\|_{2}^{2}\right]\\&=&
\frac{1}{4n}(1-(\lambda_{1}^{2}+\ldots+\lambda_{d}^{2})).
\end{eqnarray*}
Since by assumption $\|\Sigma\|_{\mathrm{op}}\leq\frac{1}{2}\mathrm{Tr}(\Sigma)=\frac{1}{2}$, we have $\lambda_{j}\leq\frac{1}{2}$ for all $j$. So $\lambda_{1}^{2}+\ldots+\lambda_{d}^{2}\leq\frac{1}{2}(\lambda_{1}+\ldots+\lambda_{d})=\frac{1}{2}$. So
\begin{equation}\label{lbfirstpart}
\mathbb{E}\left[W_{1,1}\left(\mu,\frac{1}{2n}\sum_{i=1}^{n}(\delta_{X_{i}}+\delta_{-X_{i}})\right)^{2}\right]\geq\frac{1}{8n}.
\end{equation}
This proves that the left hand side of (\ref{mainsamplinglbonerestateeq}) is at least twice the first term on the right hand side. We now move to the second part of the proof. The second term on the right hand side of (\ref{mainsamplinglbonerestateeq}) is larger than the first term precisely when $\frac{\mathrm{Tr}(\Sigma)}{n\|\Sigma\|_{\mathrm{op}}}<1$, or equivalently, $\frac{1}{n}<\lambda_{1}$. So we may assume this in the rest of the proof.

Consider the pushforward measure $(e_{1})_{\#}\mu$ of $\mu$ by the map $\langle\cdot,e_{1}\rangle$. Note that
\[(e_{1})_{\#}\mu(\{-1\})=\frac{1}{2}\lambda_{1},\quad(e_{1})_{\#}\mu(\{1\})=\frac{1}{2}\lambda_{1},\quad(e_{1})_{\#}\mu(\{0\})=1-\lambda_{1}.\]
We have
\begin{eqnarray}\label{lbsecondpart}
W_{2,1}\left(\mu,\frac{1}{2n}\sum_{i=1}^{n}(\delta_{X_{i}}+\delta_{-X_{i}})\right)&\geq& W_{2}\left((e_{1})_{\#}\mu,\,\frac{1}{2n}\sum_{i=1}^{n}(\delta_{\langle X_{i},e_{1}\rangle}+\delta_{-\langle X_{i},e_{1}\rangle})\right)\\&=&
W_{2}\left(\mathrm{abs}_{\#}(e_{1})_{\#}\mu,\,\frac{1}{n}\sum_{i=1}^{n}\delta_{|\langle X_{i},e_{1}\rangle|}\right),\nonumber
\end{eqnarray}
where $\mathrm{abs}_{\#}(e_{1})_{\#}\mu(\{1\})=\lambda_{1}$ and $\mathrm{abs}_{\#}(e_{1})_{\#}\mu(\{0\})=1-\lambda_{1}$. (See the beginning of the proof of Lemma \ref{2to1}.) Moreover, the random variables $|\langle X_{i},e_{1}\rangle|$, for $i=1,\ldots,d$, are i.i.d.~with this distribution. Thus, the probability measure $\frac{1}{n}\sum_{i=1}^{n}\delta_{|\langle X_{i},e_{1}\rangle|}$ is supported on only two points $0$ and $1$ with the mass at $1$ being $\frac{1}{n}$ times a $\mathrm{binom}(n,\lambda_{1})$ random variable, which we denote by $Y$. So we have
\[W_{2}\left(\mathrm{abs}_{\#}(e_{1})_{\#}\mu,\,\frac{1}{n}\sum_{i=1}^{n}\delta_{|\langle X_{i},e_{1}\rangle|}\right)^{2}=\left|\frac{1}{n}Y-\lambda_{1}\right|.\]
As explained above, we may assume that $\frac{1}{n}\leq\lambda_{1}$. Also by assumption, $\|\Sigma\|_{\mathrm{op}}\leq\frac{1}{2}\mathrm{Tr}(\Sigma)$ so $\lambda_{1}\leq\frac{1}{2}\leq 1-\frac{1}{n}$. Therefore, $\frac{1}{n}\leq\lambda_{1}\leq1-\frac{1}{n}$. With $\lambda_{1}$ in this range, by \cite[Theorem 1]{Berend},
\[\mathbb{E}\left|\frac{1}{n}Y-\lambda_{1}\right|\geq\frac{1}{\sqrt{2}}\left(\mathbb{E}\left|\frac{1}{n}Y-\lambda_{1}\right|^{2}\right)^{\frac{1}{2}}=\frac{\sqrt{n\lambda_{1}(1-\lambda_{1})}}{n\sqrt{2}}=\sqrt{\frac{\lambda_{1}(1-\lambda_{1})}{2n}}\geq\frac{1}{2}\sqrt{\frac{\lambda_{1}}{n}}.\]
This proves that the left hand side of (\ref{mainsamplinglbonerestateeq}) is at least twice the second term on the right hand side. Together with (\ref{lbfirstpart}), this completes the proof.
\end{proof}

\noindent{\bf Acknowledgement:} The author is grateful to Ramon van Handel,\\
Nikita Zhivotovskiy and Sloan Nietert for some useful discussions.


\begin{thebibliography}{00}
\bibitem{AZ} P.~Abdalla and N.~Zhivotovskiy, Covariance estimation: Optimal dimension-free guarantees for adversarial corruption and heavy tails,  Journal of the European Mathematical Society, 2024.
\bibitem{Adamczak} R.~Adamczak, A.~E.~Litvak, A.~Pajor and N.~Tomczak-Jaegermann, Quantitative estimates of the convergence of the empirical covariance matrix in log-concave ensembles, Journal of the American Mathematical Society, 23:535-561, 2009
\bibitem{Adamczak2} R.~Adamczak, A.~E.~Litvak, A.~Pajor and N.~Tomczak-Jaegermann, Sharp bounds on the rate of convergence of the empirical covariance matrix, Comptes Rendus. Math\'ematique 349.3-4 (2011): 195-200.
\bibitem{Bartl1} D.~Bartl and S.~Mendelson, Structure preservation via the Wasserstein distance, arXiv preprint arXiv:2209.07058 (2022).
\bibitem{Berend} D.~Berend and A.~Kontorovich, A sharp estimate of the binomial mean absolute deviation with applications, Statistics \& Probability Letters 83.4 (2013): 1254-1259.
\bibitem{Bobkov} S.~Bobkov and M.~Ledoux, One-dimensional empirical measures, order statistics, and Kantorovich transport distances, Vol. 261. No. 1259. American Mathematical Society, 2019.
\bibitem{Bonet} C.~Bonet, P.~Berg, N.~Courty, F.~Septier, L.~Drumetz, and M.-T.~Pham, Spherical sliced-Wasserstein, In International Conference on Learning Representations, 2023.
\bibitem{Bonneel}  N.~Bonneel, J.~Rabin, G.~Peyr\'e, and H. Pfister, Sliced and Radon Wasserstein barycenters of measures, Journal of Mathematical Imaging and Vision, 1(51):22-45, 2015.
\bibitem{Carriere} M.~Carri\`ere, M.~Cuturi, and S.~Oudot, Sliced Wasserstein kernel for persistence diagrams, In International Conference on Machine Learning (ICML), 2017.
\bibitem{Deshpande1} I.~Deshpande, Z.~Zhang, and A.~G.~Schwing, Generative modeling using the sliced
Wasserstein distance, In Proceedings of the IEEE Conference on Computer Vision and Pattern Recognition (CVPR), 2018.
\bibitem{Deshpande2} I.~Deshpande, Y-T.~Hu, R.~Sun, A.~Pyrros, N.~Siddiqui, S~ Koyejo, Z.~Zhao, D.~Forsyth, and A.~G.~Schwing Max-sliced Wasserstein distance and its use for GANs, In CVPR, pages 10648-10656, 2019.
\bibitem{Fournier} N.~Fournier and A.~Guillin, On the rate of convergence in Wasserstein distance of the empirical measure, Probability Theory and Related Fields 162.3-4 (2015): 707-738.
\bibitem{Haagerup} U.~Haagerup, The best constants in the Khintchine inequality, Studia Math. 70 (1981), no. 3, 231-283 (1982).
\bibitem{Kolouri1} S.~Kolouri, P.~E.~Pope, C.~E.~Martin, and G. K. Rohde. Sliced Wasserstein autoencoders. In International Conference on Learning Representations, 2018.
\bibitem{Koltchinskii} V.~Koltchinskii and K.~Lounici, Concentration inequalities and moment bounds for sample covariance operators. Bernoulli, 23:110-133, 2014.
\bibitem{Latala} R.~Lata{\l}a and K.~Oleszkiewicz, On the best constant in the Khinchin-Kahane inequality, Studia Mathematica 109.1 (1994): 101-104.
\bibitem{LT} M.~Ledoux and M.~Talagrand. Probability in Banach Spaces: isoperimetry and processes, Springer Science \& Business Media, 2013.
\bibitem{Lin2020} T.~Lin, C.~Fan, N.~Ho, M.~Cuturi and M.~Jordan, Projection robust Wasserstein distance and Riemannian optimization. Advances in Neural Information Processing Systems, 33 (2020): 9383-9397.
\bibitem{Lin} T.~Lin, Z.~Zheng, E.~Chen, M.~Cuturi, M.~I.~Jordan, On projection robust optimal transport: Sample complexity and model misspecification, In International Conference on Artificial Intelligence and Statistics  PMLR 2021.
\bibitem{Mendelson} S.~Mendelson and G.~Paouris, On the singular values of random matrices, Journal of the European Mathematical Society, 16:823-834, 2014.
\bibitem{Nadjahi} K.~Nadjahi, A.~Durmus, U.~Simsekli and R.~Badeau, Asymptotic guarantees for learning generative models with the sliced-Wasserstein distance. Advances in Neural Information Processing Systems, 32 (2019).
\bibitem{NHAm} K.~Nguyen and N.~Ho, Amortized projection optimization for sliced Wasserstein generative models, Advances in Neural Information Processing Systems 35 (2022): 36985-36998.
\bibitem{NHEnergy} K.~Nguyen and N.~Ho, Energy-based sliced Wasserstein distance, Advances in Neural Information Processing Systems 36 (2024).
\bibitem{Nietert} S.~Nietert, Z.~Goldfeld, R.~Sadhu and K.~Kato, Statistical, robustness, and computational guarantees for sliced Wasserstein distances, Advances in Neural Information Processing Systems, 35, 28179-28193 (2022).
\bibitem{NR} J.~Niles-Weed and P.~Rigollet, Estimation of Wasserstein distances in the spiked transport model, Bernoulli 28.4 (2022): 2663-2688.
\bibitem{Olea} J.~Olea, C.~Rush, A.~Velez, and J.~Wiesel, The out-of-sample prediction error of the square-root-LASSO and related estimators, arXiv preprint arXiv:2211.07608, 2022.
\bibitem{Paty} F.-P.~Paty and M.~Cuturi, Subspace robust Wasserstein distances, In ICML, pages 5072-5081, 2019.
\bibitem{Quellmalz} M.~Quellmalz, R.~Beinert, G.~Steidl, Sliced optimal transport on the sphere, Inverse Problems 39 (10), 105005, 2023
\bibitem{Raab} M.~Raab and A.~Steger, ``Balls into bins"-A simple and tight analysis, International Workshop on Randomization and Approximation Techniques in Computer Science. Berlin, Heidelberg: Springer Berlin Heidelberg, 1998.
\bibitem{Rabin} J.~Rabin, G.~Peyr\'e, J.~Delon, and M.~Bernot, Wasserstein barycenter and its application to texture mixing, In International Conference on Scale Space and Variational Methods in Computer Vision, pages 435-446. Springer, 2011.
\bibitem{Rudelson} M.~Rudelson, Random vectors in the isotropic position, Journal of Functional Analysis 164.1 (1999): 60-72.
\bibitem{SVCov} N.~Srivastava and R.~Vershynin, Covariance estimation for distributions with $2+\epsilon$ moments, Annals of Probability 41 (2013), 3081-3111.
\bibitem{Talagrand} M.~Talagrand, The generic chaining: upper and lower bounds of stochastic processes. Springer Science \& Business Media, 2005.
\bibitem{Tikhomirov} K.~Tikhomirov, Sample covariance matrices of heavy-tailed distributions, International Mathematics Research Notices 2018.20 (2018): 6254-6289.
\bibitem{Tropp} J.~A.~Tropp, An introduction to matrix concentration inequalities, Foundations and Trends in Machine Learning 8.1-2 (2015): 1-230.
\bibitem{Tropp2ndorder}  J.~A.~Tropp, Second-order matrix concentration inequalities. Appl. Comput. Harmon. Anal., 44(3):700-736, 2018.
\bibitem{vHnotes} R.~van Handel, Probability in high dimension, Lecture Notes (Princeton University) (2014).
\bibitem{Romanna} R.~Vershynin, Introduction to the non-asymptotic analysis of random matrices. Compressed sensing, 210-268, Cambridge University Press, Cambridge, 2012. 
\bibitem{Romansc} R.~Vershynin, How close is the sample covariance matrix to the actual covariance matrix?, Journal of Theoretical Probability 25.3 (2012): 655-686.
\bibitem{Romanbook} R.~Vershynin, High-dimensional probability: An introduction with applications in data science, Vol. 47. Cambridge University Press, 2018.
\bibitem{Wainwright} M.~J.~Wainwright, High-dimensional statistics: A non-asymptotic viewpoint. Vol. 48. Cambridge University Press, 2019.
\bibitem{WGX} J.~Wang, R.~Gao, and Y.~Xie, Two-sample test using projected Wasserstein distance, IEEE International Symposium on Information Theory (ISIT), 2021.\\
Updated version: arxiv: 2010.11970
\bibitem{Wu}  J.~Wu, Z.~Huang, D.~Acharya, W.~Li, J.~Thoma, D.~P.~Paudel, and L.~V.~Gool, Sliced Wasserstein generative models, In Proceedings of the IEEE/CVF Conference
on Computer Vision and Pattern Recognition, pages 3713-3722, 2019.
\bibitem{Zhivotovskiy} N.~Zhivotovskiy, Dimension-free bounds for sums of independent matrices and simple tensors via the variational principle, Electronic Journal of Probability 29 (2024): 1-28.
\end{thebibliography}
\end{document}